\documentclass[11pt,letterpaper]{amsart}
\usepackage[usenames,dvipsnames]{color}
\usepackage{amsthm,amsfonts,amssymb,amsmath,amsxtra}
\usepackage[all]{xy}
\SelectTips{cm}{}
\usepackage{xr-hyper}
\usepackage[pdfpagelabels,pdftex,hidelinks]{hyperref}
\hypersetup{
  colorlinks   = true, 
  urlcolor     = blue, 
  linkcolor    = blue, 
  citecolor   = green, 
pdftitle={},
  pdfauthor={},
  pdfsubject={},
  pdfkeywords={},
  breaklinks=true,
  bookmarksopen=true,
  bookmarksnumbered=true,
  pdfpagemode=UseOutlines,
  plainpages=false}

\usepackage{verbatim}
\usepackage{caption}

\usepackage{lineno}

\usepackage{mathrsfs}

\RequirePackage{xspace}
\RequirePackage{etoolbox}
\RequirePackage{varwidth}
\RequirePackage{enumitem}
\RequirePackage{tensor}
\RequirePackage{mathtools}
\RequirePackage{longtable}
\RequirePackage{multirow}

\def\ge{\geqslant}
\def\le{\leqslant}
\def\a{\alpha}

\def\G{\Gamma}

\def\D{\Delta}

\def\s{\sigma}
\def\t{\tau}
\def\th{\theta}
\def\k{\kappa}
\def\l{\lambda}
\def\z{\zeta}

\def\i{^{-1}}

\def\<{\langle}
\def\>{\rangle}

\newcommand{\fkg}{\ensuremath{\mathfrak{g}}\xspace}

\newcommand{\fkt}{\ensuremath{\mathfrak{t}}\xspace}

\newcommand{\fkL}{\ensuremath{\mathfrak{L}}\xspace}

\newcommand{\BA}{\ensuremath{\mathbb {A}}\xspace}

\newcommand{\BF}{\ensuremath{\mathbb {F}}\xspace}
\newcommand{{\BG}}{\ensuremath{\mathbb {G}}\xspace}
\newcommand{\BH}{\ensuremath{\mathbb {H}}\xspace}

\newcommand{{\BK}}{\ensuremath{\mathbb {K}}\xspace}
\newcommand{\BL}{\ensuremath{\mathbb {L}}\xspace}
\newcommand{\BM}{\ensuremath{\mathbb {M}}\xspace}

\newcommand{\BQ}{\ensuremath{\mathbb {Q}}\xspace}

\newcommand{\BT}{\ensuremath{\mathbb {T}}\xspace}
\newcommand{\BU}{\ensuremath{\mathbb {U}}\xspace}

\newcommand{\BW}{\ensuremath{\mathbb {W}}\xspace}
\newcommand{\BX}{\ensuremath{\mathbb {X}}\xspace}

\newcommand{\BZ}{\ensuremath{\mathbb {Z}}\xspace}

\newcommand{\CA}{\ensuremath{\mathcal {A}}\xspace}

\newcommand{\CE}{\ensuremath{\mathcal {E}}\xspace}
\newcommand{\CF}{\ensuremath{\mathcal {F}}\xspace}
\newcommand{\CG}{\ensuremath{\mathcal {G}}\xspace}
\newcommand{\CH}{\ensuremath{\mathcal {H}}\xspace}

\newcommand{\CL}{\ensuremath{\mathcal {L}}\xspace}

\newcommand{\CO}{\ensuremath{\mathcal {O}}\xspace}

\newcommand{\CT}{\ensuremath{\mathcal {T}}\xspace}
\newcommand{\CU}{\ensuremath{\mathcal {U}}\xspace}

\newcommand{\CX}{\ensuremath{\mathcal {X}}\xspace}

\newcommand{\Ad}{{\mathrm{Ad}}}

\DeclareMathOperator{\charac}{char}

\newcommand{\GL}{\mathrm{GL}}

\DeclareMathOperator{\Hom}{Hom}

\newcommand{\id}{\ensuremath{\mathrm{id}}\xspace}

\DeclareMathOperator{\Nm}{Nm}

\DeclareMathOperator{\ord}{ord}

\newcommand{\red}{\ensuremath{\mathrm{red}}\xspace}

\DeclareMathOperator{\Spec}{Spec}

\DeclareMathOperator{\tr}{tr}

\def\pr{{\rm pr}}
\def\tPhi{\widetilde \Phi}

\def\tD{\widetilde \Delta}
\def\ind{{\rm ind}}
\def\Nm{{\rm Nm}}
\def\aff{{\rm aff}}
\def\bx{{\mathbf x}}
\def\brG{\breve G}
\def\brU{\breve U}

\newtheorem{theorem}{Theorem}
\newtheorem{proposition}[theorem]{Proposition}
\newtheorem{lemma}[theorem]{Lemma}
\newtheorem {conjecture}[theorem]{Conjecture}
\newtheorem{corollary}[theorem]{Corollary}

\theoremstyle{definition}
\newtheorem{definition}[theorem]{Definition}
\newtheorem{condition}[theorem]{Condition}
\newtheorem{example}[theorem]{Example}

\newtheorem{remark}[theorem]{Remark}

\numberwithin{equation}{section}
\numberwithin{theorem}{section}

\setitemize[0]{leftmargin=*,itemsep=\the\smallskipamount}
\setenumerate[0]{leftmargin=*,itemsep=\the\smallskipamount}

\renewcommand{\to}{%
   \ifbool{@display}{\longrightarrow}{\rightarrow}%
   }
\let\shortmapsto\mapsto
\renewcommand{\mapsto}{%
   \ifbool{@display}{\longmapsto}{\shortmapsto}%
   }
\newlength{\olen}
\newlength{\ulen}
\newlength{\xlen}
\newcommand{\xra}[2][]{%
   \ifbool{@display}%
      {\settowidth{\olen}{$\overset{#2}{\longrightarrow}$}%
       \settowidth{\ulen}{$\underset{#1}{\longrightarrow}$}%
       \settowidth{\xlen}{$\xrightarrow[#1]{#2}$}%
       \ifdimgreater{\olen}{\xlen}%
          {\underset{#1}{\overset{#2}{\longrightarrow}}}%
          {\ifdimgreater{\ulen}{\xlen}%
             {\underset{#1}{\overset{#2}{\longrightarrow}}}
             {\xrightarrow[#1]{#2}}}}%
      {\xrightarrow[#1]{#2}}
   }
\makeatother
\newcommand{\xyra}[2][]{%
   \settowidth{\xlen}{$\xrightarrow[#1]{#2}$}%
   \ifbool{@display}%
      {\settowidth{\olen}{$\overset{#2}{\longrightarrow}$}%
       \settowidth{\ulen}{$\underset{#1}{\longrightarrow}$}%
       \ifdimgreater{\olen}{\xlen}%
          {\mathrel{\xymatrix@M=.12ex@C=3.2ex{\ar[r]^-{#2}_-{#1} &}}}%
          {\ifdimgreater{\ulen}{\xlen}%
             {\mathrel{\xymatrix@M=.12ex@C=3.2ex{\ar[r]^-{#2}_-{#1} &}}}
             {\mathrel{\xymatrix@M=.12ex@C=\the\xlen{\ar[r]^-{#2}_-{#1} &}}}}}%
      {\mathrel{\xymatrix@M=.12ex@C=\the\xlen{\ar[r]^-{#2}_-{#1} &}}}%
   }
\makeatletter
\newcommand{\xla}[2][]{%
   \ifbool{@display}%
      {\settowidth{\olen}{$\overset{#2}{\longleftarrow}$}%
       \settowidth{\ulen}{$\underset{#1}{\longleftarrow}$}%
       \settowidth{\xlen}{$\xleftarrow[#1]{#2}$}%
       \ifdimgreater{\olen}{\xlen}%
          {\underset{#1}{\overset{#2}{\longleftarrow}}}%
          {\ifdimgreater{\ulen}{\xlen}%
             {\underset{#1}{\overset{#2}{\longleftarrow}}}
             {\xleftarrow[#1]{#2}}}}%
      {\xleftarrow[#1]{#2}}
   }
\newcommand{\isoarrow}{%
   \ifbool{@display}{\overset{\sim}{\longrightarrow}}{\xrightarrow\sim}%
   }

\newcommand{\sm}{{\,\smallsetminus\,}}
\newcommand\cool{\overline{\mathbb{Q}}_\ell}
\newcommand\rar{ \rightarrow }
\newcommand\tar{ \twoheadrightarrow }

\newcommand{\obF}{\overline{\BF}_q}
\newcommand{\bfx}{{\mathbf x}}

\usepackage{upgreek}
\makeatletter
\newcommand{\colim@}[2]{%
  \vtop{\m@th\ialign{##\cr
    \hfil$#1\operator@font lim$\hfil\cr
    \noalign{\nointerlineskip\kern1.5\ex@}#2\cr
    \noalign{\nointerlineskip\kern-\ex@}\cr}}%
}
\newcommand{\colim}{%
  \mathop{\mathpalette\colim@{\rightarrowfill@\textstyle}}\nmlimits@
}
\newcommand{\prolim@}[2]{%
  \vtop{\m@th\ialign{##\cr
    \hfil$#1\operator@font lim$\hfil\cr
    \noalign{\nointerlineskip\kern1.5\ex@}#2\cr
    \noalign{\nointerlineskip\kern-\ex@}\cr}}%
}
\newcommand{\prolim}{%
  \mathop{\mathpalette\colim@{\leftarrowfill@\textstyle}}\nmlimits@
}
\makeatother

\begin{document}

\title{The cohomology of $p$-adic Deligne-Luszitg schemes of Coxeter type}

\author{Alexander B. Ivanov}
\address{Fakult\"at f\"ur Mathematik, Ruhr-Universit\"at Bochum, D-44780 Bochum, Germany.}
\email{a.ivanov@rub.de}

\author{Sian Nie}
\address{Academy of Mathematics and Systems Science, Chinese Academy of Sciences, Beijing 100190, China}

\address{ School of Mathematical Sciences, University of Chinese Academy of Sciences, Chinese Academy of Sciences, Beijing 100049, China}
\email{niesian@amss.ac.cn}

\begin{abstract}
We determine the cohomology of the closed Drinfeld stratum of $p$-Deligne--Lusztig schemes of Coxeter type attached to arbitrary inner forms of unramified groups over a local non-archimedean field. We prove that the corresponding torus weight spaces are supported in exactly one cohomological degree, and are pairwisely non-isomorphic irreducible representations of the pro-unipotent radical of the corresponding parahoric subgroup. We also prove that all Moy--Prasad quotients of this stratum are maximal varieties, and we investigate the relation between the resulting representations and Kirillov's orbit method.
\end{abstract}
\maketitle

\section{Introduction}

Let $k$ be a non-archimedean local field with residue characteristic $p>0$ and residue field $\BF_q$. Let $\breve k$ be the completion of the maximal unramified extension of $k$ and let $F$ denote the Frobenius automorphism of $\breve k$ over $k$. Let $G$ be a reductive group over $k$, which splits over $\breve k$. Let $T \subseteq B$ be a maximal torus and a Borel subgroup of $G$, such that $T$ splits and $B$ becomes rational over $\breve k$. Let $U$ resp. $\overline U$ denote the unipotent radical of $B$ resp. of the opposed Borel subgroup. To $G,T,U$ one can attach the space 
\begin{equation}\label{eq:XTU}
X_{T,U} = \{g \in G(\breve k) \colon g^{-1}F(g) \in \overline U \cap FU \},
\end{equation}
which is a variant of the $p$-adic Deligne--Lusztig spaces from \cite{Ivanov_DL_indrep}. Then $X_{T,U}$ has the structure of an ind-(perfect scheme) over $\overline\BF_q$.
Moreover, $X_{T,U}$ is endowed with an action of the locally compact group $G(k) \times T(k)$, so that its $\ell$-adic cohomology realizes smooth $G(k)$-representations, parametrized by smooth characters of $T(k)$, very much in the style of Deligne--Lusztig theory \cite{DeligneL_76}. Recently, the $\ell$-adic cohomology of these and closely related spaces was extensively studied (especially when $T$ is elliptic) and related with the local Langlands correspondences. See, for example, \cite{CI_loopGLn,ChanO_21} for the relation with the type-theoretic construction of J.-K. Yu \cite{Yu_01} and the related work of Kaletha and others (see e.g. \cite{Kaletha_19}). On the other hand, see \cite[\S9]{CI_loopGLn}, \cite{Feng_24} for relations with Fargues--Scholze's and Zhu's geometric local Langlands \cite{FarguesScholze,Zhu_20}. In this article we continue the study of geometry and cohomology of $X_{T,U}$.

\smallskip

Assume that $(T,U)$ is a Coxeter pair (see \S\ref{sec:Coxeter_pairs}). In particular, $T$ is elliptic and the apartment of $T$ in the reduced affine building of $G$ over $k$ consists of one point. Bruhat--Tits theory attaches to this point a parahoric model $\CG$ of $G$ over the integers $\CO_k \subseteq k$ with connected special fiber. Let $\CO$ denote the integers of $\breve k$. It was shown in  \cite{Ivanov_Cox_orbits, Nie_23} that $X_{T,U} \cong \coprod_{G(k)/\CG(\CO_k)} g X$, where
\begin{equation}\label{eq:X}
X = \{ g \in \CG(\CO) \colon g^{-1}F(g) \in (\overline\CU \cap F\CU)(\CO)\}
\end{equation}
is a perfect affine $\obF$-scheme with $\CG(\CO_k) \times \CT(\CO_k)$-action, and where we denote by $\CT,\CU \subseteq \CG$ the closures of $T,U$. Cohomology of $X_{T,U}$ is then obtained by compactly inducing that of $X$. 

\smallskip

There is a fibration $X \rar X_{0+}$ over a Deligne--Lusztig variety $X_{0+}$ of the reductive quotient $\BG_{0+} = (\CG \otimes_{\CO_k} \BF_q)_{\rm red}$ of the special fiber of $\CG$. The variety $X_{0+}$ admits a natural stratification by locally closed subschemes. The stratification of $X$ obtained by pulling it back was first considered in \cite{CI_DrinfeldStrat} (for $\GL_n$ and inner forms) resp. in \cite[\S6.2]{ChanO_21} (in general) and called the Drinfeld stratification there. There is a (in full generality only conjectural) relation between the cohomologies of $X$ and of the strata, see \cite[Theorem 5.1]{CI_loopGLn}, \cite[Conjecture 7.2.1]{CI_DrinfeldStrat}, \cite[Conjecture 6.5]{ChanO_21}

\smallskip

The cohomology of the unique closed stratum is very interesting, and seems to be the most accessible one. When $G$ is an inner form of $\GL_n$, its cohomology as a $\CG(\CO_k) \times \CT(\CO_k)$-representation was determined in \cite[Theorem 6.1.1]{CI_DrinfeldStrat}, the case of division algebras (where the closed stratum coincides with the whole scheme $X$) being already handled in \cite{Chan_siDL}. The main goal of the present article is to extend these results to all $G$, thus giving a full account of the cohomology of the closed stratum. As a consequence we also produce a rich supply of maximal varieties in the sense of \cite{BoyarchenkoW_16} associated with groups other than $\GL_n$. Our second goal is to investigate how this cohomology relates to representations obtained via Kirillov's orbit method, see below. 

\smallskip

To state our main result, let $\CG^+$ be the pro-unipotent radical of $\CG$ and let $\CT^+,\CU^+$ be the closures of $T,U$ in $\CG^+$. Then the closed stratum is a disjoint union of finitely many copies of the affine perfect scheme 
\begin{equation}\label{eq:Xplus}
Y = \{ g \in \CG^{+}(\CO) \colon g^{-1}F(g) \in (\overline\CU \cap F\CU^+)(\CO)\}.
\end{equation}
with $\CG^{+}(\CO_k) \times \CT^{+}(\CO_k)$-action. As $Y$ is infinite-dimensional, it has no reasonable cohomology with compact support. We could remedy this by working with quotients of $Y$ attached to Moy--Prasad quotients of $\CG^{+}$ (and on the technical level we will do precisely this). However, it seems most natural to state our results in terms of the homology functor $f_\natural$, which is the left adjoint of $f^\ast$, introduced in \cite{IvanovM} in the schematic context following the approach of \cite[VII.3]{FarguesScholze} (see \S\ref{sec:homology_in_general} for more details). Let therefore $H_i(Y,\cool)$ denote the homology groups of the complex $f_\natural\cool$, where $f \colon Y \rar \Spec \obF$ is the structure map. If $\chi$ is a smooth character $\CT^+(\CO_k) \rar \cool^\times$, we also have the $\chi$-weight part $f_\natural\cool[\chi]$ of $f_\natural\cool$. Let $N \geq 1$ be the smallest positive integer with $F^N U = U$. Then $Y$ has an obvious $\BF_{q^N}$-rational structure. 

\begin{theorem}\label{thm:intro} Suppose that $(T, U)$ is a Coxeter pair. For a smooth character $\chi \colon \CT^{+}(\CO_k) \rar \cool^\times$ the following hold. 
\begin{itemize}
\item[(1)] Assume that $p$ satisfies Condition \ref{hypo}\footnote{this holds if if the derived group of $G$ is simply connected and $p\geq 5$; it also always holds if $p$ does not divide the order of the Weyl group of $G$.}. The homology of $f_\natural \cool[\chi]$ is non-vanishing in precisely one degree $s_{\chi} \geq 0$. 

\smallskip

\item[(2)] Assume that $p$ satisfies Condition \ref{hypo}. The Frobenius $F^N$ acts in the space $H_{s_\chi}(Y,\cool)[\chi]$ as multiplication by the scalar $(-1)^{s_\chi} q^{s_\chi N/2}$. In particular, all Moy--Prasad quotients of $Y$ are $\BF_{q^N}$-maximal varieties.

\smallskip

\item[(3)] For varying $\chi$, $H_{s_\chi}(Y,\cool)[\chi]$ runs through pairwise non-isomor\-phic irreducible smooth $\CG^{+}(\CO_k)$-representations.
\end{itemize}
\end{theorem}

This theorem follows from Theorems \ref{main}, \ref{thm:irred} and Corollary \ref{max} (where for part (1) the discussion of \S\ref{sec:homology_in_general} and Corollary \ref{cor:transition_maps_Yr} apply). We determine the integer $s_\chi$ explicitly in terms of the Howe factorization of $\chi$, see Corollary \ref{cor:formal_degree}.

\smallskip

In fact, the same proof of Theorem \ref{thm:irred}, combined with Remark \ref{minimal}, shows that the statement (3) of Theorem \ref{thm:intro} is true if $(T, U)$ is a minimal elliptic pair, see \S\ref{sec:Coxeter_pairs}. This partially motivates us to propose the following conjecture.
\begin{conjecture}
    Theorem \ref{thm:intro} holds for all minimal elliptic pairs $(T, U)$.
\end{conjecture}

Using parts (1),(2) of the theorem along with a fixed point formula of Boyarchenko \cite[Lemma 2.12]{Boyarchenko_12}, we give the following representation-theoretic interpretation of the integer $s_\chi$, generalizing \cite[Lemma 8.1]{CI_loopGLn}. 

\begin{corollary}\label{cor:dimension}
If Condition \ref{hypo} holds for $p$, then $\dim_{\cool} H_{s_\chi}(Y,\cool)[\chi] = q^{s_\chi N/2}$.
\end{corollary}

This corollary is proven in \S\ref{sec:traces}. More generally, we obtain a trace formula for any element of $\CG^+(\CO_k)$ on $H_{s_\chi}(Y,\cool)[\chi]$ in terms of geometric points of (a Moy--Prasad quotient of) $Y$, see Proposition \ref{cor:traces}.

\smallskip

To apply our main result to the cohomology of $X_{T,U}$ (in the style of \cite{CI_loopGLn}) it is necessary to study the relation between the cohomology of $X$ and of the closed stratum (\cite[Conjecture 6.5]{ChanO_21}); this will be considered in a follow-up work. Once this is done, our results, combined with the main results of \cite{ChanO_21} and \cite{DudasI_20} (see \cite[Corollary 1.0.2]{DudasI_20}), would give geometric approaches to some representation-theoretic questions. For example, Corollary \ref{cor:dimension} allows a purely geometric proof of the formal degree formulas for many supercuspidal representations (note that an algebraic computation is given in the recent work of Schwein \cite{Schwein_24}).

\smallskip

The second goal of this article is to formulate and verify in a special case a conjecture about the relation of the homology of $Y$ with Kirillov's orbit method for the pro-$p$-group $\CG^+(\CO_k)$, whenever the latter applies. 
Namely, by a theory of Lazard, a uniform pro-$p$-group (resp. a $p$-group of nilpotence class $<p$) $\Gamma$ is completely described by its $\BZ_p$-Lie algebra (resp. finite Lie ring) $\fkg$ via an exponential map, see \cite[\S2]{BoyarchenkoD_10}. Kirillov's orbit method establishes a natural bijection between smooth irreducible representations of $\Gamma$ and adjoint $\Gamma$-orbits in the dual $\fkg^\ast = {\rm Hom}_{\rm cont}(\fkg,\cool^\times)$, see \cite{BoyarchenkoS_08}, characterized by a trace formula. Often it happens that $\CG^+(\CO_k)$ (resp. its Moy--Prasad quotient) is a uniform pro-$p$-group (resp. $p$-group of nilpotence class $<p$).
In this case the natural question to determine the adjoint orbit corresponding to $H_{s_\chi}(Y,\cool)[\chi]$ arises. In Conjecture \ref{conj:relation_to_orbit_method} we make this precise. We verify this conjecture for the finite $p$-group $\{ g \in \GL_2(\BF_q[\varpi]/\varpi^3) \colon g \equiv 1 \mod \varpi\}$ if $q$ is odd.

\smallskip

Finally, we complete the task of comparing the spaces $X_{T,U}$ from \eqref{eq:XTU} with the $p$-adic Deligne--Lusztig spaces from \cite{Ivanov_DL_indrep}, when $(T,U)$ is a Coxeter pair. This was done for classical groups in \cite[Proposition 5.12]{Ivanov_Cox_orbits}, and in \S\ref{sec:comparing_toXwb} we prove it for general $G$. To achieve this, we need to extend the loop version of twisted Steinberg's cross-section (see \cite[3.6]{HeL_12}, \cite[Proposition 5.3]{Ivanov_Cox_orbits} and \cite{Malten_21}) to non-classical groups, see Proposition \ref{St}. Note that this result is also used in the proof of Theorem \ref{thm:intro}(3). 

\subsection*{Acknowledgements} The first author gratefully acknowledges the support of the German Research Foundation (DFG) via the Heisenberg program (grant nr. 462505253). He would like to thank Moritz Firsching for answering his questions related to SAGE. The second author would like to thank Miaofen Chen, Xuhua He, Jilong Tong and Weizhe Zheng for answering his questions and for helpful discussions.

\section{Notation and setup}

\subsection{General notation}\label{sec:general_notation}

Throughout the article we let $\breve k/k$ with integers $\CO_k \subseteq \CO$, residue field extension $\obF/\BF_q$, and Frobenius $F$ be as in the introduction. We denote by $\varpi$ a uniformizer of $k$.

Given a $\BF_q$-algebra $R$, let ${\rm Perf}_R$ be the category of perfect $R$-algebras. For $R \in {\rm Perf}_{\BF_q}$, let $W(R)$ be the ring of $p$-typical Witt vectors of $R$, and put $\BW(R) = W(R) \otimes_{\BZ_p} \CO_k$ if ${\rm char}(k) = 0$, resp. $\BW(R) = R[\![\varpi]\!]$ otherwise. In particular, $\BW(\BF_q) = \CO_k$ and $\BW(\obF) = \CO$. Let $[\cdot] \colon R \rar \BW(R)$ be the Teichm\"uller lift if ${\rm char}(k) = 0$, resp. $[x] = x$ if ${\rm char}(k) > 0$. 

Let $\CX$ be any $\CO$-scheme and let $X$ be any $\breve k$-scheme. We will abbreviate 
\[\breve \CX := \CX(\CO) \quad \text{ and } \quad \breve X = X(\breve k). \] 
Suppose that $\CX$ is affine and of finite type over $\CO$. We regard the set $\breve \CX$ as a perfect affine scheme $\BX$ over $\obF$, so that $\BX(\obF) = \breve \CX$. More precisely, one puts $\BX = L^+\CX$, where $L^+\CX \colon {\rm Perf}_{\obF} \rar {\rm Sets}$, $L^+\CX(R) = \CX(\BW(R))$ is the functor of positive loops, see e.g. \cite[\S2.5]{CI_MPDL} for details. We always will identify the scheme $\BX$ with the set $\breve \CX$ of its geometric points. If $\CX$ is defined over $\CO_k$, $\BX$ has a natural $\BF_q$-structure, corresponding to the $F$-action on $\breve \CX$. Moreover, the set 
\[ 
\BX(\BF_q) = \breve \CX^F = \CX(\CO_k)
\] 
has a natural structure of a profinite set. Similarly, if $X$ is affine of finite type over $\breve k$, then we regard $\breve X$ as an ind-(perfect affine scheme) over $\obF$ via the loop functor $LX \colon {\rm Perf}_{\obF} \ni R \mapsto X(\BW(R)[p^{-1}])$, and the same claim about $\BF_q$-structure holds, except that now $\breve X^F$ is only locally profinite.

\subsection{Group-theoretic setup} We fix a reductive group $G$ defined over $k$ and split over $\breve k$. We fix a $k$-rational, $\breve k$-split maximal torus $T$ of $G$, we denote by $N_G(T)$ its normalizer. Its Weyl group $W= N_G(T)/T$ is a finite \'etale group scheme over $k$ becoming constant over $\breve k$. We identify $W$ with the set of its $\breve k$-points, endowed with the action of $F$. We denote by $X_\ast(T)$, $X^\ast(T)$ the groups of (co)characters of $T_{\breve k}$, equipped with natural $F$-actions, and by $\langle,\rangle \colon X^\ast(T) \times X_\ast(T) \rar \BZ$ the natural $W$- and $F$-equivariant pairing. We denote by $N$ the order of $F$ as an automorphism of $X_\ast(T)$. 

We fix a Borel subgroup $T \subseteq B \subseteq G$ defined over $\breve k$, and we denote by $U$ the unipotent radical of $B$. Denote by $\Phi \subseteq X^\ast(T)$ the set of roots of $T$ in $G$, and by $\Phi^+$ resp. $\Phi^-$ the subset of positive roots corresponding to $U$ resp. $\overline U$. For each $\alpha \in \Phi$, let $U_\alpha \cong \BG_{a,\breve k}$ denote the corresponding root subgroup.

\subsection{Filtration of the torus and affine roots}\label{sec:filtrations_affine_roots}
Let $\CT$ denote the connected N\'eron model of $T$. Let $\breve T^0$ be the maximal bounded subgroup of $\breve T$. Then $\CT(\CO) = \breve T^0$.
Moreover, for $r \in \BZ_{\geq 0}$, 
\[
\breve T^r = \{t \in \breve T^0 \colon \ord_{\varpi}(\chi(t) - 1) \geq r \, \forall \chi \in X^\ast(T) \}
\]
defines a descending separated filtration on $\breve T$. For each $r$ one has an isomorphism 
\[ 
V := X_\ast(T) \otimes \obF \stackrel{\sim}{\longrightarrow} \breve T^r /\breve T^{r+1}, \quad \lambda \otimes x \mapsto \lambda(1 + [x]\varpi^r).
\] 

Fix some (e.g., hyperspecial) point ${\bf x}_0$ in the apartment $\CA_{T,\breve k}$ of $T$ in the reduced affine building of $G$ over $\breve k$. Let 
\[
\tPhi_\aff = \{\a+m: x \mapsto -\a(x - \bx_0) + m; \alpha \in \Phi, m \in \BZ \} \cong \Phi \times \BZ
\]
be the set of affine roots.
Let $\widetilde\Phi = \Phi_{\rm aff} \sqcup \BZ_{\geq 0}$ be the (enlarged) set of affine roots of $T$ in $G$. For an affine root $\a+m$, we have the corresponding subgroup $\brU_{\a+m} \subseteq \brU$. For $m \in \BZ_{\geq 0}$, the corresponding root subgroup is $\breve T^m$. There is an action of $F$ on $\widetilde \Phi$, such that $F\brU_{\a+m} = \brU_{F(\a+m)}$.

\subsection{Parahoric model and Moy--Prasad quotients} Assume that $T$ is elliptic. Then the apartment of $T$ in the reduced affine building of $G$ over $k$ consists of precisely one point $\bfx$. We denote by $\CG$ the parahoric $\CO_k$-model of $G$ with connected special fiber attached to $\bfx$, and by $\CG^+$ its pro-unipotent radical. 

If $H \subseteq G$ is a closed subgroup, then we denote by $\CH$ the closure of $H$ in $\CG$.\footnote{Note that for $H=T$ there is no conflict of notation with \S\ref{sec:filtrations_affine_roots} as the closure of $T$ in $\CG$ is the connected N\'eron model of $T$ by \cite[4.7.4 Lemma and 8.2 Corollary]{Yu_02}.} Similarly, we denote by $\CT^+$ the closure of $T$ in $\CG^+$.

Note that $\breve \CG$ (resp. $\breve \CG^+$) is generated by all $\breve U_{f}$ with $f \in \widetilde\Phi$ satisfying $f(\bfx) \geq 0$ (resp. $f(\bfx) > 0$), and that $\breve \CG/\breve \CG^+$ is naturally isomorphic to the reductive quotient of the special fiber of $\CG$.

For any $h = r$ or $h = r+$ with $r \in \BZ_{\geq 0}$, Moy--Prasad have defined in \cite{MoyP_94} the normal $F$-stable subgroup $\breve \CG^h \subseteq \breve\CG$ generated by all $\breve U_f$ with $f \in \widetilde\Phi$ satisfying $f(\bfx) \geq h$. Note that $\breve \CG = \CG(\CO)^0$ and $\breve \CG^+ = \CG(\CO)^{0+}$. There is a smooth $\BF_q$-group scheme $\BG_r$ with 
\[ \BG_r(\obF) = \breve \CG/\breve\CG^r. \] 
It has the subgroup $\BG_r^+ = \breve \CG^+/\breve \CG^r$, and the set of affine roots appearing in $\BG_r^+$ is
\[
\widetilde \Phi_r^+ = \{f \in \widetilde\Phi \colon 0 < f(\bfx) < r \}.
\]
According with \S\ref{sec:general_notation}, we have also the $\BF_q$-groups $\BG$ and $\BG^+$ such that $\BG(\obF) = \breve \CG$ and $\BG^+(\obF) = \breve \CG^+$. Note that $\BG = \prolim_r\BG_r$ and $\BG^+ = \prolim_r \BG^+_r$.

Note that any of the subgroups $H = T,B,U, \dots$ of $G$ defines a closed subgroup $\BH_r \subseteq \BG_r$ (resp. $\BH \subseteq \BG$) by first taking the closure $\CH \subseteq \CG$ of $H$, and then letting $\BH_r(\obF)$ be the image of the map $\CH(\CO) \rar \CG(\CO) \rar \BG_r(\obF)$. Similarly, $H$ defines a closed subgroup $\BH_r^+ \subseteq \BG_r^+$ (and $\BH^+ \subseteq \BG^+$). Note that if $F^sH = H$ for some $s\geq 1$, then $\BH_r,\BH_r^+$ are defined over $\BF_{q^s}$.

\subsection{Coxeter pairs}\label{sec:Coxeter_pairs}
Let $c \in W$ be the unique element such that $FB = {}^c B$. Then for any lift $\dot c$ of $c$, $\Ad({\dot c})\i \circ F: \brG \to \brG$ fixes the pinning $(T,B)$, hence defines an automorphism $\sigma_W$ of the Coxeter system $(W,S)$. We call $(T,B)$ (and $(T,U)$) a \emph{Coxeter pair} if $c$ is a Coxeter element in the Coxeter triple $(W,S,\sigma_W)$, that is, if a(ny) reduced expression of $c$ contains precisely one element from each $\sigma_W$-orbit on $S$. More generally, $(T, U)$ is called a \emph{minimal elliptic pair} if $c$ is of minimal length in its $\s_W$-twisted conjugacy class. We have implications $(T,B)$ Coxeter $\Rightarrow$ $(T,B)$ minimal elliptic $\Rightarrow$ $T$ is elliptic.

\smallskip

We define 
\[
\Delta := \Phi^- \cap F\Phi^+ 
\]
Note that if $(T,B)$ is Coxeter, then each $F$-orbit in $\Phi$ has length exactly $N$ and intersects the set in precisely one element, see e.g. \cite[\S7]{Steinberg_65}. In particular, $\#\Delta$ is equal to the semisimple rank of $G$, $\Phi / \<c \s_W\> \cong \D$ and $\#\Phi = N \cdot \#\Delta$.

\subsection{A condition on $p$} \label{sec:condition_on_p}
Assume that $T$ is elliptic. We will prove Theorem \ref{main} under the following condition on the characteristic $p$ of $\BF_q$, which is satisfied if $p$ does not divide the order of the Weyl group of $G$. 

\begin{condition} \label{hypo}
The characteristic $p$ of $\BF_q$ is not a torsion prime for $\Phi$ (see \cite[Definition 1.3]{Steinberg_75}) and $p$ does not divide $\#\pi_1(M_{\rm der})$ for any $F$-stable Levi subgroup $M$ containing $T$. Here $M_{\rm der}$ denotes the derived subgroup of $M$.
\end{condition}

Note the all torsion primes for $\Phi$ are $\leq 5$. Note that the second part of this condition holds for all $p$ when $G_{\rm der}$ is simply connected. Let $P = P(G,T)$ denote the set of primes, for which this condition does not hold. If $G$ is of type $A_n$, then $P \subseteq \{\ell \text{ prime } \colon \text{$\ell$ divides $n$} \}$. If $G$ is of type $B_n$ or $C_n$ with $n$ even, then $P \subseteq \{2\}$. If $G$ is of type $B_n$ or $C_n$ with $n$ odd, then $P \subseteq \{2\} \cup \{\ell \text{ prime } \colon \text{$\ell$ divides $n$}\}$. If $G$ is of type $D_n$, then $P \subseteq \{\ell \text{ prime } \colon \text{$ \ell < n$}\}$.

We will use this condition in the proof of Theorem \ref{main} by applying the following lemma to derived subgroups of various $F$-stable unramified twisted Levi subgroups of $G$ containing $T$. Recall $V = X_\ast(T) \otimes \obF$ from \S\ref{sec:filtrations_affine_roots} and consider the following norm map 
\[
\Nm_N: V \to V, ~ v \mapsto v + F(v) + \cdots + F^{N-1}(v).
\]

\begin{lemma} \label{surj}
Suppose that $G$ is semisimple and $p$ does not divide $\#\pi_1(G)$. Then $V^F = \Nm_N(\BZ\Phi^\vee \otimes \BF_{q^N})$, where $\Phi^\vee$ is the set of coroots. 
\end{lemma}
\begin{proof}
    By assumption we have $\BZ\Phi^\vee \otimes \BF_{q^N} = X_*(T) \otimes \BF_{q^N}$. Hence \[V^F = \Nm_N(V^{F^N}) = \Nm_N(X_*(T) \otimes \BF_{q^N}) = \Nm_N(\BZ\Phi^\vee \otimes \BF_{q^N})\] as desired.
\end{proof}

\subsection{Homology}\label{sec:homology_in_general}

For a morphism $Y \rar Z$ of perfect $\BF_p$-schemes and a coefficient ring $\Lambda$, which we assume to be either $\cool$ or $\overline\BF_\ell$ here, in \cite{IvanovM} the left adjoint $f_\natural \colon D_{\blacksquare}(Y,\Lambda) \rar D_{\blacksquare}(Z, \Lambda)$ of $f^\ast$ on unramified solid sheaves is constructed. Readers feeling uncomfortable with the use $f_\natural$, could just regard \eqref{eq:homology_of_Y} as a definition (which is well-behaved because of \eqref{eq:transition_maps_isom}). Assume that $Z = \Spec \overline\BF_q$, in which case we get the $\Lambda$-module 
\[
H_i(Y,\Lambda) := H^{-i}f_\natural\Lambda.
\] 
Assume now that $Y = \prolim_r Y_r$ with all $f_r \colon Y_r \rar \Spec \overline\BF_q$ perfections of smooth morphisms of dimension $d_r$. Assume that there are compatible actions of finite groups $\Gamma_r$ on $Y_r$, inducing an action of $\Gamma = \prolim_r \Gamma_r$ on $Y$. Let $\chi \colon \Gamma \rar \Lambda^\times$ be a smooth character. There is some $r_\chi \geq 0$ such that for each $r\geq r_\chi$, $\chi$ factors through a character of $\Gamma_r$ again denoted $\chi$. Assume that for all $r \geq r_\chi$ the map 
\begin{equation}\label{eq:transition_maps_isom}
f_{r !}\Lambda[\chi][2(d_r - d_{r_\chi})]  \rar  f_{r_\chi !}\Lambda[\chi].
\end{equation}
is an isomorphism. 
As $f_{\natural}$ commute with cofiltered limits of schemes, we have
\[
f_\natural \Lambda[\chi] = \prolim_r f_{r\natural}\Lambda[\chi] = \prolim_r f_{r !}\Lambda[\chi][2d_r] = f_{r_\chi !}\Lambda[\chi][2d_{r_\chi}],
\]
where the second equality holds because $f_r$ is smooth (and hence $f_{r \natural} = f_{r!}[2d_r]$) and the last equality is by \eqref{eq:transition_maps_isom}. With other words,
\begin{equation}\label{eq:homology_of_Y}
H_i(Y,\Lambda)[\chi] = H^{2d_r - i}_c(Y_{r_\chi},\Lambda)[\chi] \quad \text{for all $r \geq r_\chi$}.
\end{equation}

\section{Steinberg's cross-section}\label{sec:steinberg}

The following proposition is a variant of \cite[3.6, 3.14]{HeL_12}, generalizing \cite[Proposition 5.3]{Ivanov_Cox_orbits}.

\begin{proposition}\label{St} Suppose $(T,U)$ is a Coxeter pair. Then the map $(x, y) \mapsto x\i y F(x)$ induces isomorphisms: 
    
    (1) $(\BU_r \cap F\BU_r) \times (\overline\BU_r \cap F\BU_r) \cong  F\BU_r$;

    \smallskip
    
    (2) $\BU_r \times (\overline\BU_r \cap F\BU_r) \cong \BU_r F\BU_r = \BU_r (\overline\BU_r \cap F\BU_r) \cong \BU_r \times (\overline\BU_r \cap F\BU_r)$.

Moreover, the analogous statements also hold with $\BU_r$ replaced by $\BU_r^+$ or by $\breve \CU$ or by $\breve \CU^+$ or by $\breve U$.
\end{proposition}

\begin{remark} \label{minimal}
    Using a different approach, Malten \cite{Malten_21} shows that Proposition \ref{St} holds for all minimal elliptic pairs $(T, U)$. We will not use this result in the paper.
\end{remark}

We use this result for $\BU_r^+$ in \S\ref{sec:irreducibility} to deduce the irreducibility of $H_{s_\chi}(Y,\cool)[\chi]$, and for $\breve U$ in \S\ref{sec:XTU} to prove the isomorphism of $X_{T,U}$ with the $p$-adic Deligne--Lusztig space from \cite{Ivanov_DL_indrep}.

\begin{proof} 
In any of the setups ($\BU_r,\BU_r^+,\breve \CU,\breve \CU^+,\breve U$), (1) is equivalent to (2) as in \cite[3.14]{HeL_12}, so it suffices to prove (1). By \cite[3.6]{HeL_12}, the map in (1) is always injective. 

In the setup with $\BU_r$ resp. $\BU_r^+$ the proposition follows from injectivity and \cite[Proposition 1.2(ii)]{HeL_12}, as the source and the target of the map in (1) are isomorphic to the (perfect) affine space over $\obF$ of the same finite dimension. By passing to the inverse limit over $r$, the proposition also follows in the setup with $\breve \CU,\breve \CU^+$.

It remains to handle the setup with $\breve U$, where we argue as in \cite[Proposition 5.3]{Ivanov_Cox_orbits}. By \cite[\S3.5]{HeL_12} it suffices to prove (1) for a single Coxeter element. By \cite[Lemma 5.5]{Ivanov_Cox_orbits}, it suffices to assume that the Dynkin diagram of $G$ is connected.  The cases when $G$ is classical were handled in \cite[Proposition 5.3]{Ivanov_Cox_orbits}, so it suffices to verify \cite[Lemma 5.7]{Ivanov_Cox_orbits} for the remaining types ($G_2$, $F_4$, $E_6$, $E_7$, $E_8$, ${}^3D_4$, ${}^2E_6$). That is, we must provide a filtration \[\Phi^+ = \Psi_r \supseteq \Psi_{r-1} \supseteq \dots \Psi_2 \supseteq \Psi_1 = \Phi^+ \cap F\i(\Phi^-),\] such that for each $i$, $\Psi_i$ and $\Psi_i \sm \Psi_1$ are closed under addition; the implication $\alpha,\beta \in \Psi_i$, $\alpha+\beta \in \Phi \Rightarrow \alpha+\beta \in \Psi_{i-1}$ holds for all $i>1$; and for all $i$, $F(\Psi_i \sm \Psi_1) \subseteq \Psi_i$. We do this using an algorithm implemented in SAGE \cite{sagemath}. It even turns out that it is always possible to arrange that $\#(\Psi_{i+1} \sm \Psi_i) = 1$. Our algorithm is explained in Appendix \ref{sec:appendix_sage}.
\end{proof}

\section{Review and some properties of $X_{T,U}$}
\label{sec:XTU}

Let $X_{T,U}$ be as in \eqref{eq:XTU}. Here we recall/prove some facts about it. As we explain below, if $T$ is elliptic there is an equivariant map $X_{T,U} \rar \dot X_{\dot w}(b)$ into a certain $p$-adic Deligne--Lusztig space from \cite[\S8]{Ivanov_DL_indrep}. If $(T,U)$ is Coxeter and if $G$ is classical in the sense of \cite[Definition 5.1]{Ivanov_Cox_orbits}, it was shown in \cite[Proposition 5.12]{Ivanov_Cox_orbits} that this map is an isomorphism. We prove in this section that this holds for all $G$.

\subsection{Comparison with the definition in \cite{Ivanov_DL_indrep}}  \label{sec:comparing_toXwb}

Assume that $T$ is elliptic. Assume that $G$ admits a (necessarily unique) unramified inner form $G_0$ over $k$ (the general case easily reduces to this by using a derived embedding of $G$ into a group with connected center). 
Then one can choose 
\begin{itemize}
\item a $k$-rational pinning $(T_0, B_0 = T_0 U_0)$ of $G_0$ with Weyl group $(W_0,S_0)$, 
\item an elliptic element $w \in W_0$,
\item a lift $\dot w \in N(T_0)(\breve k)$,
\end{itemize}
such that there is a $\breve k$-rational isomorphism $G \stackrel{\sim}{\rar} G_0$, identifying $T,B,U,W$ with $T_0,B_0,U_0,W_0$, and $F$ with $Ad(\dot w) \circ \sigma$ as an automorphism of $\breve G \cong \breve G_0$. Let $b \in \breve G$. In \cite[\S8]{Ivanov_DL_indrep}, the \emph{$p$-adic Deligne--Lusztig space} attached to the datum $(G_0, \dot w, b)$ is defined as the arc-sheaf on perfect $\obF$-schemes, 
\[
\dot X_{\dot w}(b) = \{x \in L(G_0/U_0) \colon x^{-1}b\sigma(x) \in L(U_0 \dot w U_0) \},
\]
where $L(\cdot)$ is the perfect loop functor as in \S\ref{sec:general_notation}. Note that $(g,t) \colon x \mapsto gxt$ defines an action of the locally profinite group $G(k) \times T(k)$ on this arc-sheaf, see \cite[\S8]{Ivanov_DL_indrep} for details. This seems to be a natural definition, most similar to classical Deligne--Lusztig varieties.

Note that $w$ equals the relative position of $U$ with $FU$. Thus $(T,U)$ is a Coxeter (resp. minimal elliptic) pair if and only if $w$ is a Coxeter (resp. minimal elliptic) element. 

Identifying $G$ with $G_0$ via the given isomorphism, we have the composition
\begin{align}\label{eq:comparison_XTU_pDL}
\nonumber X_{T,U} &\rar \{g \in \breve G_0 \colon g^{-1}\dot w\sigma(g) \in \dot w \breve U_0 \}/(\breve U_0 \cap {}^{w}\breve U_0)\\ &\stackrel{\sim}{\rar} \dot X_{\dot w}(\dot w),
\end{align}
given by $g \mapsto g (\breve U_0 \cap {}^{w}\breve U_0) \mapsto g\breve U_0$.
Just as was done in \cite[Proposition 5.12]{Ivanov_DL_indrep} for classical groups, we deduce from Proposition \ref{St}:

\begin{corollary}
Assume $(T,U)$ is a Coxeter pair. Then the map \eqref{eq:comparison_XTU_pDL} is a $G(k)\times T(k)$-equivariant isomorphism.
\end{corollary}

\subsection{Integral decomposition of $X_{T,U}$}

A priori, $X_{T,U}$ is a huge ind-scheme, which is hard to control. However, in the Coxeter case, it has the following decomposition. Let $X$ be as in \eqref{eq:X} and note that $X \subseteq X_{T,U}$ is a closed subscheme. Surprisingly, it is also open and the following holds.

\begin{theorem}[\cite{Ivanov_Cox_orbits},\cite{Nie_23}] \label{decomposition}
    Suppose $(T,U)$ is a Coxeter pair and let $X$ be as in \eqref{eq:X}. Then there is a decomposition 
    \[
    X_{T,U} = \bigsqcup_{g \in G(k)/\CG(\CO_k)} g X.
    \] 
    In particular, $X_{T,U}$ is a disjoint union of affine perfect $\overline \BF_q$-schemes.
\end{theorem}

This reduces the study of the cohomology of $X_{T,U}$ to that of $X$.

\subsection{Drinfeld stratification}\label{sec:Drinfeld_strat}

In this subsection, the ellipticity assumption on $T$ can be dropped. Note that the projection $\BG \rar \BG_{0+}$ restricts to a projection 
\[X \rar X_{0+} = \{g \in \BG_{0+} \colon g^{-1}F(g) \in \overline\BU_{0+} \cap F\BU_{0+}\}\] 
over (a variant of) a classical Deligne--Lusztig variety. Let $\fkL$ denote the set of all twisted Levi subgroups of $\BG_{0+}$ containing $\BT_{0+}$. For any $\BL_{0+} \in \fkL$ we have the locally closed $\BG_{0+}^F \times \BT_{0+}^F$-stable closed subscheme
\[
X_{0+}^{(\BL_{0+})} = \{ g \in \BG_{0+} \colon g^{-1}F(g) \in \BL_{0+} \cap \overline\BU_{0+} \cap F\BU_{0+} \}
\]
Pulling back to $X$ we obtain a closed subscheme $X^{(\BL_{0+})} \subseteq X$. Following \cite{CI_DrinfeldStrat} and \cite[\S6.2]{ChanO_21}, we then call
\[
X^{(\BL_{0+})} \sm \bigcup_{\BL_{0+}' \subseteq \BL_{0+} \in \fkL} X^{(\BL'_{0+})}
\]
a \emph{Drinfeld stratum} of $X$. This defines a finite and locally closed stratification of $X$. Its has a unique minimal/closed stratum $X^{(\BT_{0+})}$.

\begin{lemma}\label{lm:decomposition_closed_Drinfeld_stratum}
With $Y$ as in \eqref{eq:Xplus}, we have $X^{(\BT_{0+})} = \bigsqcup_{g \in \BG_{0+}^F/\BT_{0+}^F} g (X \cap \BT \BG^+)  = \bigsqcup_{g \in \BG_{0+}^F} g Y$.
\end{lemma}
\begin{proof}
The first equality is \cite[Lemma 3.3.3]{CI_DrinfeldStrat}. As $\BT_{0+} \cap \overline\BU_{0+} \cap F\BU_{0+} = 1$, the image of $X^{(\BT_{0+})}$ under $X \rar X_{0+}$ is contained in the finite subset $\BG_{0+}^F \subseteq X_{0+}$. By exploiting the $\BG^F$-action on $X$ and the surjectivity of $\BG^F \rar \BG_{0+}^F$, each fiber is a translate of $Y$.
\end{proof}

In the rest of the article we consider $Y$ and its cohomology. To approximate $Y$, consider for any $r \in \BZ_{>0}$ the affine perfect $\obF$-scheme
\[
Y_r = \{ g \in \BG_r^+ \colon g^{-1}F(g) \in \overline\BU_r\cap F\BU_r^+ \},
\]
equipped with $(\BG_r^+)^F \times (\BT_r^+)^F$-action, so that $Y = \prolim_r Y_r$. Similarly, we have the schemes $X_r^{(\BL_{0+})} \subseteq \BG_r$ approximating $X^{(\BL_{0+})}$.

Recall the set $\Delta$ from \S\ref{sec:Coxeter_pairs}. Let $\Phi^\red$ denote the set of those $\alpha \in \Phi$ for which $\alpha(\bfx) \in \BZ$, and let $\Delta^{\rm red} = \Phi^{\rm red} \cap \Delta$.

\begin{lemma}\label{lm:properties_of_Yr}
The scheme $Y_r$ is the perfection of an affine smooth scheme of dimension $r \cdot \#\Delta - \#\Delta^{\red} = \frac{1}{N} (r \cdot \#\Phi - \Phi^{\rm red})$.
\end{lemma}
\begin{proof}  The last equality follows from the last sentense of \S\ref{sec:Coxeter_pairs}. Note that the if $\alpha \in \Delta^{\rm red}$ (resp. $\alpha \in \Delta \sm \Delta^{\rm red}$), then there are precisely $r-1$ (resp. precisely $r$) affine roots with vector part $\alpha$ appearing in $\overline\BU_r \cap F\BU_r^+$. Thus $\overline\BU_r \cap F\BU_r^+$ is isomorphic to the perfection of $\BA^{r\cdot\#\Delta - \#\Delta^{\red}}_{\obF}$.
Note that $Y_r$ is the pullback of $\overline\BU_r \cap F\BU_r^+$ under the Lang map $g\mapsto g^{-1}F(g)$ of $\BG_r^+$. By \cite[Lemma A.26]{Zhu_17}, there is a smooth algebraic group $\BH$ over $\BF_q$ with perfection $\BG_r^+$. Let $\BW \subseteq \BH$ be the (reduced) closed subgroup whose perfection is $\overline\BU_r \cap F\BU_r^+$. In particular, $\BW$ is necessarily isomorphic to $\BA^{r\cdot\#\Delta - \#\Delta^{\red}}_{\obF}$. Let $Y_r'$ be the pullback of $\BW$ under the Lang map of $\BH$. As perfection commutes with limits, $Y_r$ is the perfection of $Y_r'$. As the Lang map is \'etale, the claim follows.
\end{proof}

\section{The minimal Drinfeld stratum}

In this section we assume that $\Phi / \<c \s_W\> \cong \D$ and $\#\Phi = N \cdot \#\D$, where $c$ and $\s_W$ are as in \S\ref{sec:Coxeter_pairs}. This condition is satisfied if $(T, U)$ is a Coxeter pair.

\smallskip

We will study the geometric and cohomological properties of $Y_r$ for $r \in \BZ_{>0}$.
To this end, we will study Deligne-Lusztig type constructions for various subquotient groups of $\BG_r^+$.

\subsection{A total order on affine roots}
For $f \in \tPhi$ we write $\a_f \in \Phi \sqcup \{0\}$ and $m_f \in \BZ$ such that $f = \a_f + m_f$. Let $\CO_f$ be the $F$-orbit of $f$. 

Let $\Phi_{\aff}^+$ (resp. $\tPhi^+$) be the set of affine roots $f \in \Phi_{\aff}$ (resp. $f \in \tPhi$) such that $f(\bx) > 0$. Note that $\tPhi^+ = \Phi_\aff^+ \sqcup \BZ_{\ge 1}$ and $\tPhi = \Phi_{\aff}^+ \sqcup \tPhi^0 \sqcup -\tPhi_{\rm aff}^+$. Here $\tPhi^0 = \{f \in \tPhi; f(\bx) = 0\}$.

Recall that $\D = \Phi^- \cap F\Phi^+$. Set $\Delta_\aff^+ = (\D \times \BZ) \cap \Phi_\aff^+$ and $\tD^+ = \D_{\aff}^+ \sqcup \BZ_{\ge 1}$.
\begin{lemma} \label{orbit}
    The map $f \mapsto \CO_f$ induces a bijection $\tD^+ \cong \tPhi^+ / \<F\>$.
\end{lemma}
\begin{proof}
This follows from our assumption on $(T, U)$ in this section.
\end{proof}

\begin{definition} \label{order}
We define a linear order $\le$ on $\tPhi^+$ such that
\begin{itemize}
    \item $f < f'$ if either (1) $f(\bx) < f'(\bx)$ or (2) $f(\bx) = f'(\bx)$, $f \in \BZ_{\ge 1}$ and $f' \in \D_\aff^+$;

    \item if $f_1, f_2 \in \tD^+$ such that $f_1 < f_2$, then $f_1' < f_2'$ for any $f_1' \in \CO_{f_1}$ and any $f_2' \in \CO_{f_2}$.

   \item $f < F(f) < \cdots < F^{N-1}(f)$ for $f \in \D_\aff^+$.
\end{itemize}   
\end{definition}

Let $f \in \tD^+$. We denote by $f+$ and $f-$ the descendant and the ascendant of $f$ in $\tD^+ \sqcup \{0\}$ respectively such that $0 = f-$ if $f = \min \tD^+$. Set $\tPhi^f = \{f' \in \tPhi^+; f' \ge f\}$.

\subsection{The variety $Y^A_B$} We fix an integer $r \in \BZ_{\ge 1}$. Let $\BG_r^+ = \breve\CG^{0+} / \breve\CG^r$ and let $\tPhi_r^+ = \{f \in \tPhi; 0 < f(\bx) < r\}$ be the set of affine roots appearing in $\BG_r^+$. 

Let $f \in \tPhi^+$. If $f \in \Phi_\aff^+$, we take $\BA_f = \BG_a$ and define $u_f: \BA^1 \to \BG_r^+$ by $x \mapsto U_{\a_f}([x]\varpi^{m_f})$ for $x \in \overline \BF_q$. If $f \in \BZ_{\ge 1}$, we take $\BA_f = V := X_*(T) \otimes \overline \BF_q$ and define $u_f: \BA_f \to \BG_r^+$ by $\l \otimes x \mapsto \l(1 + [x] \varpi^{n_f})$ for $\l \in X_*(T)$ and $x \in \overline \BF_q$. 

We define an abelian group $\BA[r] = \prod_{f \in \tPhi_r^+} \BA_f$. Then we have an isomorphism of varieties \[u: \BA[r] \overset \sim \to \BG_r^+, \quad (x_f)_f \mapsto \prod_f u_f(x_f),\] where the product is taking with respect to the linear order $\le$ on $\tPhi^+$.

Let $E \subseteq \tPhi_r^+$. We set $\BA_E = \prod_{f \in E} \BA_f$, which is viewed as a subgroup group of $\BA[r]$ in the natural way. Denote by $p_E: \BA[r] \to \BA_E$ the natural projection. Using the identification $u: \BA[r] \cong \BG_r^+$ we define \[\pr_E = u \circ p_E \circ u\i: \BG_r^+ \to u(\BA_E).\] For $f \in \tPhi_r^+$ we put $p_f = p_{\{f\}}$ and $\pr_f = \pr_{\{f\}}$. By abuse of notation, we will identify $\pr_f: \BG_r^+ \to u(\BA_f)$ with $u\i \circ \pr_f: \BG_r^+ \to \BA_f$ freely according to the context.

Let $A, B \subseteq \tPhi^+$ be two subsets. We set \[A + B = \{f + f' \in \tPhi; f \in A, f' \in B\}.\] We say $A$ is closed if $A + A \subseteq A$ and $A + \BZ_{\ge 0} = A$. In this case, we denote by $\BG_r^A \subseteq \BG_r^+$ the subgroup generated by $u(\BA_f)$ for $f \in A$.

Suppose that $\tPhi^r \subseteq A, B \subseteq \tPhi^+$ are $F$-stable and closed such that $B \subseteq A$ and $A + B \subseteq B$. Then $\BG_r^B$ is a normal subgroup of $\BG_r^A$. The isomorphism $u: \BA[r] \overset \sim \to \BG_r^+$ restricts to an isomorphism $u_{A:B}: \BA_{A \setminus B} \overset \sim \to \BG_r^A / \BG_r^B$. So we get an embedding \[s_{A:B} = u \circ u_{A:B}\i : \BG_r^A/\BG_r^B \to \BG_r^+.\] We define \[Y_B^A = \{g \in \BG_r^A; g\i F(g) \in (\overline\BU_r \cap F\BU_r) \BG_r^B\} / \BG_r^B \subseteq \BG_r^A / \BG_r^B,\] which admits a natural action by $(\BG_r^A)^F \times (\BT_r^+ \cap \BG_r^A)^F$. Let $\chi: (\BT_r^+ \cap \BG_r^A)^F \to \overline \BQ_\ell^\times$ be a character. We denote by $H_c^i(Y_B^A, \overline \BQ_\ell)[\chi]$ the $\chi$-weight space of the $(\BT_r^+)^F$-action on $H_c^i(Y_B^A, \overline \BQ_\ell)$. For $f \in \tPhi_r^+$ we define \[\pi^{A:B}_f = u\i \circ \pr_{\CO_f} \circ L \circ s_{A:B}: \BG_r^A/\BG_r^B \to \BA_{\CO_f}.\] Here, for any $F$-stable sub-quotient group of $\BG_r^+$, we always denote by $L$ the Lang's self-map given by $g \mapsto g\i F(g)$.

\begin{proposition} \label{fiber}
    Let $\tPhi^r \subseteq A, B \subseteq \tPhi^+$ be $F$-stable and closed. Let $f \in B$ and $C = B \setminus \CO_f$. Suppose that $\tPhi^r \subseteq C$ is closed, $C + A \subseteq C$ and $\CO_f + A \subseteq C$. Then

    (1) if $f \in \D_\aff^+$, then the map $\psi = (q_f, \pr_f): Y_C^A \cong Y_B^A \times \BA_f$ is an isomorphism;

    (2) if $f \in \BZ_{\ge 1}$ (in which case $\BA_{\CO_f} = \BA_f = V$), then there is a Cartesian diagram \[ \xymatrix{
    Y_C^A \ar[d]_{q_f} \ar[r]^{\pr_f} &  \BA_f \ar[d]_{-L} \\
    Y_B^A \ar[r]^{\pi^{A:B}_f} &  \BA_f.} \] 
    
    Here $q_f$ denotes the natural projection.
\end{proposition}
\begin{proof} By assumption, the map $u$ induces an identification $\BA_{\CO_f} \cong \BG_r^B/\BG_r^C$ as abelian groups. Moreover,

(a) $\BG_r^B/\BG_r^C$ lies in the center of $\BG_r^A/\BG_r^C$.  

Assume that $f \in \D_\aff^+$. We define a morphism $\phi: Y_B^A \times \BA_f \to Y_C^A$ as follows. Let $(g, y) \in Y_B^A \times \BA_f$. Write $\pi^{A:B}_f(g) = (z_i)_{1 \le i \le N} \in \BA_{\CO_f}$ with each $z_i \in \BA_{F^i(f)}$. We define \[\phi(g, y) = s_{A:B}(g) u(y) F(u(y)) \cdots F^{N-1}(u(y)) \prod_{1 \le i \le N-1} u(z_i) F(u(z_i)) \cdots F^{N-i-1}(u(z_i)).\] By (a) one checks that \[\phi(Y_B^A \times \BA_f) \subseteq Y_C^A\] and $\psi \circ \phi = \id$. Let $g \in Y_C^A$ and set $g' = \phi(\psi(g)) \in Y_C^A$. Then $\psi(g) = \psi(g')$, that is, $g\i g' \in \BA_{\CO_f \setminus \{f\}} \subseteq \BG_r^B / \BG_r^C$. As $g, g' \in Y_C^A$, it follows by (a) that \[L(g\i g') = L(g)\i L(g') \in \BA_f \subseteq \BG_r^B / \BG_r^C.\] Hence $g = g'$ by Lemma \ref{zero}. So $\phi \circ \psi = \id$ and (1) is proved.

Assume that $f \in \BZ_{\ge 1}$. As both vertical maps in the diagram are finite \'etale $V^F$-torsors, it suffices to show that the square commutes. Let $g \in Y_C^A$. Write $s_{A:C}(g) = u(x) u(y)$ with $x \in \BA_{A \setminus B}$ and $y \in \BA_f$. Then $\pr_f(g) = y$ and $q_f(g) = u(x)$. As $f \in \BZ_{\ge 1}$ and $g \in Y_C^A$, we have $\pr_f (L(s_{A:C}(g)) = 0 \in \BA_f$. Using (a) one computes that
\begin{align*} 
\pi_f^{A:B}(q_f(g)) &= \pr_f(L(u(x))) \\ &= \pr_f(L(s_{A:C}(g) u(y)\i)) \\
&= \pr_f(L(s_{A:C}(g)) L(u(y)\i)) \\ &= \pr_f(L(s_{A:C}(g))) \pr_f(L(u(y)\i)) \\
&= L(y\i) \\ &= - L(y).
\end{align*}
So (2) is proved.
\end{proof}

\begin{lemma} \label{zero}
    Let $f \in \D_\aff^+$ and let $x = (x_i)_{0 \le i \le N-1} \in \BA_{\CO_f}$ with each $x_i \in \BA_{F^i(f)}$ such that $L(x) \in \BA_f$. Then $x_i = F^i(x_0)$ for $1 \le i \le N-1$. In particular, (1) $L(x) = F^N(x_0) - x_0$ and (2) $x = 0$ if and only if $x_0 = 0$.
\end{lemma}
\begin{proof}
    By definition we have \[L(x) = F(x) - x  = \sum_{i=0}^{N-1}F(x_{i-1}) - x_i \in \BA_{\CO_f},\] from which the lemma follows. 
\end{proof}

\subsection{Main result} For $f' \le f \in \tD^+$ we set $\BG_f^+ = \BG_r^+ / \BG_r^{\tPhi^f}$, $Y_f = Y^{\tPhi^+}_{\tPhi^f}$, $\BT_f = \BT_{\lceil f \rceil}$ and $\BT_f^{f'} = \ker(\BT_f \to \BT_{f'})$, where $\tPhi^f = \{f' \in \tPhi^+; f' \ge f\}$ and $\lceil f \rceil = \min \{n \in \BZ_{\ge 1}, n \ge f\}$. Note that $\BT_{f+}^{f}$ is nontrivial if and only if $f \in \BZ_{\ge 1}$, in which case $\BT_{f+}^f \cong V = X_*(T) \otimes \overline \BF_q$. 
\begin{theorem} \label{main}
    Assume that $p$ satisfies Condition \ref{hypo}. Let $f \in \tD^+$ and let $\chi: (\BT_f^+)^F \to \overline \BQ_\ell^\times$ be a character. Then there exists $s = s_{f, \chi} \in \BZ_{\ge 0}$ such that \[H_c^i(Y_f, \overline \BQ_\ell)[\chi] \neq 0 \Longleftrightarrow i = s,\] on which $F^N$ acts by multiplication by $(-1)^s q^{sN/2}$.
\end{theorem}

After necessary preparations we prove Theorem \ref{main} in \S\ref{sec:proof_main}. We compute the cohomological degree $s_{\chi,r}$ explicitly in terms of the Howe factorization of $\chi$ in \S\ref{sec:computation_coho_degree}. A variety over a finite field is called \emph{maximal} in \cite{BoyarchenkoW_16}, if its number of rational points attains the Weil--Deligne bound given by its Betti numbers.

\begin{theorem} \label{torsor}
    Let $f: Z \to Y$ be an \'etale $\G$-torsor, where $\G$ is a finite group. Let $\Lambda$ be a ring. Assume that either $\Lambda$ is finite, or $Z, Y$ are irreducible and geometrically unibranch. Then \[f_! (\Lambda) = \bigoplus_\rho \rho \otimes \CE_\rho,\] where $\rho$ ranges over irreducible representations of $\G$ and $\CE_\rho$ is a local system on $Y$.
\end{theorem}
\begin{proof}
The category of locally constant $\Lambda$-sheaves on $Z_{\rm et}$ is equivalent to the category of continuous $\pi_1(Z)$-representations on finite $\Lambda$-modules \cite[0GIY, 0DV5]{StacksProject}. The same holds for $Y$ and the functor $f_! = f_\ast$ correspond to induction of representations. Thus $f_!(\Lambda)$ corresponds to the $\pi_1(Y)$-representation $\ind_{\pi_1(Z)}^{\pi_1(Y)} 1_{\pi_1(Z)}$, which is equal to the inflation along $\pi_1(Y) \twoheadrightarrow \G$ of the regular $\G$-representation. The latter decomposes as $\bigoplus_{\rho \in {\rm Irr}(\G)} \rho^{\oplus \dim(\rho)}$. Thus, if $\CE_\rho$ denotes the local system on $Y$ corresponding to the inflation of $\rho$, we deduce $f_!(\Lambda) \cong \oplus_{\rho} \CE_\rho^{\oplus \dim(\rho)}$.
\end{proof}

\begin{proposition} \label{equi}
    Let $\G$ be a finite group. Suppose that $Z$ and $Y = Z \times \BA^1$ are $\G$-varieties, and the natural projection $\pi: Y \to Z$ is $\G$-equivariant. Then we have $H_c^i(Y, \overline \BQ_\ell) \cong H_c^{i-2}(Z, \overline \BQ_\ell)$ as $\G$-modules.
\end{proposition}
\begin{proof}
    It suffices to show $\pi_!(\overline \BQ_\ell) \cong \overline \BQ_\ell[-2]$ as $\G$-equivariant sheaves. Indeed, the adjunction map gives an isomorphism\[\pi_!(\overline \BQ_\ell) \cong \pi_!\pi^*(\overline \BQ_\ell) \cong \pi_!\pi^!(\overline \BQ_\ell [-2]) \cong \overline \BQ_\ell [-2]\] as $\G$-equivariant sheaves.  
\end{proof}

\subsection{Multiplicative local systems}  
Let $P$ be a commutative unipotent algebraic group defined over $\BF_q$. Then the map $\CL \mapsto t_\CL$ induces a bijection from the isomorphism classes of multiplicative local systems on $P$ to the set $\Hom(H(\BF_q), \overline \BQ_\ell^\times)$ of characters of $P(\BF_q)$. Here $t_\CL: P(\BF_q) \to \overline \BQ_\ell^\times$ is the trace-of-Frobenius function for $\CL$. See \cite[\S1.8]{BoyarchenkoD_10} for details. For $\th \in \Hom(P(\BF_q), \overline \BQ_\ell^\times)$ we denote by $\CL_\th$ the multiplicative local system corresponding to $\th$.

\begin{lemma} \label{base-change}
Let $\CL$ be a multiplicative local system on $P$. Then the base change of $\CL$ to $P_{\BF_{q^n}}$ (with $n \in \BZ_{\geq 1}$) corresponds to the character $t_\CL \circ \Nm_n$, where $\Nm_n(x) = x F_P(x) \cdots F_P^{n-1}(x)$ and $F_P$ denotes the Frobenius automorphism of $P$.    
\end{lemma}

For a character $\chi$ of $(\BT_{f+}^+)^F$ we denote by $\chi_{f+}^f$ the restriction of $\chi$ to $(\BT_{f+}^f)^F$. Proposition \ref{fiber} has the following consequence:

\begin{corollary} \label{coh}
    Let $f \in \tD^+$ and let $\chi$ be a character of $(\BT_{f+}^+)^F$.
    
    (1) if $f \in \D_\aff^+$, then $H_c^i(Y_{f+}, \overline \BQ_\ell)[\chi] \cong H_c^{i-2}(Y_f, \overline \BQ_\ell)[\chi]$;

    (2) if $f \in \BZ_{\ge 1}$, then $H_c^i(Y_{f+}, \overline \BQ_\ell)[\chi_{f+}^f] \cong H_c^i(Y_f, \pi^*(\CL_{\chi_{f+}^f}))$, and hence \[H_c^i(Y_{f+}, \overline \BQ_\ell)[\chi] \cong H_c^i(Y_f, \pi^*(\CL_{\chi_{f+}^f}))[\chi].\] Here $\pi = \pi^{\tPhi^+:\tPhi^f}_f$ and $H_c^i(Y_{f+}, \overline \BQ_\ell)[\chi_{f+}^f]$ is the $\chi_{f+}^f$-weight space of $(\BT_{f+}^f)^F$. 
\end{corollary}
\begin{proof}
    If $f \in \D_\aff^+$, by Lemma \ref{fiber} (1) we have $Y_{f+} \cong Y_f \times \BG_a$, and the natural projection $q_f: Y_{f+} \to Y_f$ respects the right actions of $(\BT_{f+}^+)^F = (\BT_f^+)^F$ on $Y_{f+}$ and $Y_f$. So the statement (1) follows from Proposition \ref{equi}.

    Now assume that $f \in \BZ_{\ge 1}$. Note that the Lang's map $L: \BT_{f+}^f \to \BT_{f+}^f$ is an \'{e}tale $(\BT_{f+}^f)^F$-torsor. It follows from Theorem \ref{torsor} that \[L_! (\overline \BQ_\ell) = \bigoplus_\th \CL_\th,\] where $\th$ ranges over characters of $(\BT_{f+}^f)^F$, and $\CL_\th$ is the multiplicative local system corresponding to $\th$. By the Cartesian diagram in Proposition \ref{fiber} (2), it follows from the base change theorem that \[ (q_f)_!(\overline \BQ_\ell) = (q_f)_! \pr_f^* (\overline \BQ_\ell) \cong \pi_f^* L_!(\overline \BQ_\ell) = \bigoplus_\th \pi_f^* \CL_\th.\] So the statement (2) follows by noticing that $(\BT_{f+}^f)^F$ acts on the sheaf $\CL_\th$ via the character $\th$.    
\end{proof}

From this we deduce:

\begin{corollary}\label{cor:transition_maps_Yr}
Let $\chi \colon \CT(\CO_k) \rar \cool^\times$ be a smooth character, which factors through $(\BT_r^+)^F$. Then for any $r_2 \geq r_1 \geq r$, the map 
$f_{r_2,!}\cool[\chi][2(\dim Y_{r_2} - \dim Y_{r_1})] \rar f_{r_1,!}\cool[\chi]$ is an isomorphism, where $f_{r_i} \colon Y_{r_i} \rar \Spec\obF$ is the structure map. With other words, \eqref{eq:transition_maps_isom} holds for the schemes $Y_r$.
\end{corollary}

\subsection{Reduction to the semisimple case} \label{subsec: derived}
Let $G' \subseteq G$ be the derived subgroup. Let $T'$ be a maximal torus of $G'$ contained in $T$. One can define the objects $Y_f' = Y_f$ for $G'$ in a similar way. 
\begin{lemma} \label{der}
    For $f \in \tD^+$ we have \[Y_f = \bigsqcup_{x \in (\BT_f^+)^F/ (\BT_f^{\prime +})^F} x Y_f' = \bigsqcup_{x \in (\BT_f^+)^F / (\BT_f^{\prime +})^F} Y_f' x\i.\] In particular, $H_c^i(Y_f, \overline \BQ_\ell) \cong \ind_{(\BT_f^{\prime +})^F}^{(\BT_f^+)^F} H_c^i(Y_f', \overline \BQ_\ell)$ as $(\BT_f^+)^F$-modules.
\end{lemma}
\begin{proof}
    Let $g \in Y_f$. Then $g\i F(g) \in \overline\BU_f^+ \cap F\BU_f^+ \subseteq \BG_f^{\prime +}$. By Lang's theorem, there exists $g' \in \BG_f^{\prime +}$ such that ${g'}\i F(g') = g\i F(g)$. So $g = (g {g'}\i) g' \in (\BG_f^+)^F Y_f'$ and hence $Y_f = (\BG_f^+)^F Y_f'$. 
    
    On the other hand, there is a natural isomorphism \[ (\BT_f^+)^F / (\BT_f^{\prime +})^F \cong (\BG_f^+)^F / (\BG_f^{\prime +})^F.\] Now it follows that \[Y_f^+ = (\BG_f^+)^F Y_f' = \bigsqcup_{x \in (\BT_f^+)^F/ (\BT_f^{\prime +})^F} x Y_f^{\prime +} = \bigsqcup_{x \in (\BT_f^+)^F / (\BT_f^{\prime +})^F} Y_f^{\prime +} x\i,\] where the last equality follows from the observation that $(\BT_f^+)^F$ normalizes $Y_f'$.
\end{proof}

\subsection{Handling jumps in the Howe factorization of $\chi$} \label{subsec:red}
We fix a positive integer $h \le r$ and a character $\chi$ of $(\BT_{h+}^+)^F $. Recall that $\BT_{h+}^h \cong \BA_h = V = X_*(T) \otimes \overline \BF_q$, and recall from \S\ref{sec:condition_on_p} the norm map
\[
\Nm_N: V \to V, ~ v \mapsto v + F(v) + \cdots + F^{N-1}(v).
\]

Using the character $\chi$ we define a root system \[\Phi_\chi = \{\a \in \Phi; \chi \circ \Nm_N (\a^\vee \otimes \BF_{q^N}) = \{1\}\}.\] Note that $\Phi_\chi$ is $F$-stable. By \cite[Lemma 3.6.1]{Kaletha_19} it is a Levi subsystem of $\Phi$ (note that by Condition \ref{hypo}, $p$ is not a torsion prime for $\Phi$).

Let $M = M_\chi \subseteq G$ be the twisted Levi subgroup generated by $T$ and $U_\a$ for $\a \in \Phi_\chi$. Let $\tPhi_M$ be the set of affine roots of $M$. We set \[D = (\D_\aff^+ \cap \Phi_h^+) \setminus \tPhi_M = \{f \in \D_\aff^+ \setminus \tPhi_M; f < h\}.\] By Lemma \ref{orbit} the map $f \mapsto \CO_f$ gives a natural bijection \[D \overset \sim \to (\tPhi_h^+ \setminus \tPhi_M) / \<F\>.\]
Let $f \in D$. As $f < h$, it follows by Definition \ref{order} that $0 < f(\bx) < h$ and hence $h-f \in \tPhi_h^+ \setminus \tPhi_M$. Hence there exists a unique affine root $f^\flat \in \D_\aff^+$ such that $-f+h \in \CO_{f^\flat}$. In particular, $f^\flat \in D$ and $f(\bx) + f^\flat(\bx) = h$. We label all the affine roots in $D$ by \[f_1, \dots, f_{m-1}, f_m = f_m^\flat, \dots, f_n = f_n^\flat, f_{m-1}^\flat, \dots, f_1^\flat\] such that \[f_1(\bx) \le \cdots \le f_{m-1}(x) \le \frac{h}{2} = f_m(\bx) = \cdots = f_n(\bx) = \frac{h}{2} \le f_{m-1}^\flat(\bx) \le \cdots \le f_1^\flat(\bx),\] $f_i < f_i^\flat$ for $1 \le i \le m-1$ and $f_{m-1}^\flat < \cdots < f_1^\flat$.

Let $1 \le i \le m$. We set $D_i^\flat = \{f_j^\flat \in D; 1 \le j \le i\}$ if $1 \le i \le m-1$ and $D_i^\flat = \{f_j^\flat; 1 \le j \le n\}$ if $i=m$. Define \[A_i = \tPhi^+ \setminus \cup_{j=1}^{i-1} \CO_{f_j}, \quad B_i = \tPhi^h \cup \bigcup_{f \in D_i^\flat} \CO_f, \quad C_{i-1} = B_{i-1} \setminus \{h\}.\] Moreover, we set $A_0 = A_1 = \tPhi^+$, $B_0 = \tPhi^h$ and $C_0 = B_0 \setminus \{h\}$. Note that $A_m = B_m \cup \tPhi_M^+$ with $\tPhi_M^+ = \tPhi_M \cap \tPhi^+$.

\begin{lemma} \label{sum}
    Let $1 \le i \le m$. Then  $A_{i-1} + A_{i-1} \subseteq A_i$, $A_i + B_i \subseteq B_{i-1}$, $\tPhi_M^+ + B_i \subseteq C_{i-1}$, $C_{i-1} + C_{i-1} \subseteq C_{i-1}$ and $A_{i+1} + B_i \subseteq C_{i-1}$, where $A_{m+1} = B_{m-1} \cup \tPhi_M^+$. In particular, $A_i$, $B_i$ and $C_{i-1}$ are $F$-stable and closed.
\end{lemma}
\begin{proof}
    We only show the second and the third inclusions. The others can be proved similarly. Let $f \in A_i$ and $f' \in B_i$ such that $f + f' \in \tPhi$.  
   
    First we assume that $f \in \tPhi_M^+$. Then $f+f' \notin \tPhi_M^+$ since $f' \notin \tPhi_M^+$. As \[(f+f')(\bx) > f'(\bx) \ge f_i^\flat(\bx) \ge h/2.\] we have $f + f' \in \cup_{f'' \in D_{i-1}^\flat} \CO_{f''} \subseteq C_{i-1} \subseteq B_{i-1}$ as desired. 
    
    Now we assume that $f \notin \tPhi_M^+$. Then $f(\bx) \ge f_i(\bx)$ and \[(f+f')(\bx) \ge f_i(\bx) + f_i^\flat(\bx) = h.\] By Definition \ref{order} we have $f + f' \in \tPhi^h \subseteq B_{i-1}$ as desired.
\end{proof}

Let $g \in \BG_r^+$, $x \in \BA[r]$ and $E \subseteq \tPhi_r^+$. We set $g_E = \pr_E(g) \in u(\BA_E)$, $x_E = p_E(x) \in \BA_E$ and $\hat x = u(x) \in \BG_r^+$. For $f \in \tPhi_r^+$ we will set $x_f = x_{\{f\}}$ and $x_{\ge f} = x_{\tPhi^f}$. We can define $g_f$ and $g_{\ge f} \in \BG_r^+$ in a similar way. By abuse of notation, we will identify $g_f \in u(\BA_f)$ with $u\i(g_f) \in \BA_f$ according to the context.
\begin{lemma} \label{pr1}
    Let $A, B \subseteq \tPhi^+$ such that $A + B \subseteq \tPhi^h$. Let $x \in \BA_A$ and $y \in \BA_B$. Then \[ (\hat y \hat x)_h = -\sum_f \a_f^\vee \otimes y_{h-f} x_f + y_h + x_h \in \BA_h = V,\] where $f$ ranges over $A$ such that $f < h-f$.
\end{lemma}
\begin{proof}
    As $A + B \subseteq \tPhi^h$, for any $f \in A$ and $f' \in B$ we have $[\hat y_{f'}, \hat x_f] = \hat y_{f'} \hat x_f \hat y_{f'}\i \hat x_f\i \in \BG_r^h$. Moreover, one computes that  
    
    (a) $[\hat y_{f'}, \hat x_f]_h = \a_f^\vee \otimes y_{f'} x_f$ if $f + f' = h$, and $[\hat y_{f'}, \hat x_f]_h= 0$ otherwise. 
    
    Assume that $y = y_{\le f'}$ and $x = x_{\ge f}$ for some $f' \in B$ and $f \in A$. We argue by induction on $f$. If $f \ge f'$, the statement is trivial. Suppose that $f < f'$. Then we have \begin{align*}
        (\hat y \hat x)_h &= (\hat y_{< f'} \hat y_{f'} \hat x_f \hat x_{> f})_h \\ &= (\hat y_{< f'} \hat x_f \hat y_{f'} [\hat y_{f'}\i, \hat x_f\i] \hat x_{> f})_h \\ &=(\hat y_{< f'} \hat x_f \hat y_{f'} \hat x_{>f} )_h + [\hat y_{f'}\i, \hat x_f\i]_h \\ &~\vdots \\ &= (\hat y_{\le f} \hat x_f \hat y_{[f+, f']} \hat x_{>f})_h + \sum_{f'' \in [f+, f']} [\hat y_{f''}\i, \hat x_f\i]_h \\ &= (\hat y_{\le f} \hat x_f)_h + (\hat y_{[f+, f']} \hat x_{>f})_h + \sum_{f'' \in [f+, f']} [\hat y_{f''}\i, \hat x_f\i]_h
    \end{align*} where $[f+, f'] = \{f'' \in \tPhi^+; f+ \le f'' \le f'\}$. Now the statement follows from (a) and induction hypothesis.
\end{proof}

\begin{lemma} \label{pr2}
    Let $1 \le i \le m$. Let $x \in \BA_{A_i}$ and $y \in \BA_{B_i}$. Assume that $x \in \BA_{A_i \cap \tPhi_M}$ or $1 \le i \le m-1$. Then $(\hat x \hat y)_h = x_h + y_h \in \BA_h$.  
\end{lemma}
\begin{proof}
    Assume that $x = x_{\le f}$ and $y = y_{\ge f'}$ for some $f \in A_i$ and $f' \in B_i$. We argue by induction on $f'$. If $f \le f'$, the statement is trivial. Assume that $f > f'$. We claim that

    (a) $f + f' > h$ if $f + f' \in \tPhi$.

    First note that $f(\bx) + f'(\bx) \ge 2f'(\bx) \ge 2f_i^\flat(\bx) \ge h$. Suppose that (a) does not hold. Then $f + f' = h$. Assume $x \in \BA_{A_i \cap \tPhi_M}$. Then $f \in \tPhi_M^+$ and $f + f' \in C_{i-1}$ by Lemma \ref{sum}, which is a contradiction. Assume $1 \le i \le m-1$. If $f \in \CO_{f_i}$, then $f' \in \CO_{f_i^\flat}$ and hence $f < f'$ by our choice that $f_i < f_i^\flat$, which is a contradiction. So $f \in A_{i+1}$ and $f + f' \in A_{i+1} + B_i \subseteq C_i$ by Lemma \ref{sum}, which is also a contradiction. So (a) is proved.

    By (a) we have $[\hat x_f\i, \hat y_{f'}\i] \in \BG_r^{h+}$. Hence
    \begin{align*} (\hat x \hat y)_h &= ((\hat x_{< f} \hat y_{f'} \hat x_f [\hat x_f\i, \hat y_{f'}\i]) \hat y_{>f'})_h \\ &= (\hat x_{<f} \hat y_{f'} \hat x_f \hat y_{>f'})_h \\ &~\vdots \\ &= (\hat x_{\le f'} \hat y_{f'} \hat x_{[f'+, f]} \hat y_{>f'})_h \\ &= (\hat x_{\le f'} \hat y_{f'})_h + (\hat x_{[f'+, f]} \hat y_{>f'})_h \\ &= (\hat x_{\le f'})_h + (\hat y_{f'})_h + (\hat x_{[f'+, f]} \hat y_{>f'})_h.
    \end{align*} Now the statement follows by induction hypothesis.
\end{proof}

We set $\pi = \pi_h^{\tPhi^+: \tPhi^h}: \BG_h^+ = \BG_r^+/\BG_r^h \to \BA_h \cong V$.
\begin{proposition} \label{factor}
    Let $1 \le i \le m$. Then there is an isomorphism \[\psi_i: Y^{A_i}_h \cong Y^{A_i}_{B_i} \times \BA_{D_i^\flat}.\] Moreover, for $(\hat x, y) \in Y^{A_i}_{B_i} \times \BA_{D_i^\flat}$ with $x \in \BA_{A_i \setminus B_i}$ we have 
    
    (1) if $1 \le i \le m-1$, then \[\pi(\psi_i\i(\hat x, y)) = \a_{f_i}^\vee \otimes (x_{f_i}^{q^N} - x_{f_i}) y_{f_i^\flat}^{q^{n_i}} + \pi(\psi_i\i(\hat x, 0)) \in V,\] where $0 \le n_i \le N-1$ such that $F^{n_i}(f_i^\flat) = -f_i + h$;

    (2) if $i=m$, then $\pi(\psi_i\i(\hat x, y)) = -\sum_{j=m}^n \a_{f_j}^\vee \otimes y_{f_j}^{q^{N/2}+1} + \pi(\psi_i\i(\hat x, 0))$. 
\end{proposition}
\begin{proof}
Without loss of generality we may assume that $m = n$. In particular, $B_i = \tPhi^h \cup \CO_{f_1^\flat} \cup \cdots \cup \CO_{f_i^\flat}$ for $0 \le i \le m$. 

By Lemma \ref{sum} we have $A_i + \CO_{f_j^\flat} \subseteq A_j + B_j \subseteq B_{j-1}$ for $1 \le j \le i \le m$. Thus by applying Proposition \ref{fiber} (1) repeatedly, we obtain an isomorphism \[\psi_i: Y^{A_i}_h = Y^{A_i}_{B_0} \cong Y^{A_i}_{B_1} \times \BA_{f_1^\flat} \cong \cdots \cong Y^{A_i}_{B_i} \times \BA_{D_i^\flat}.\]

Let $z = s_{\tPhi^+: \tPhi^h} \circ \psi_i\i(\hat x, y)$. We claim that

(a) $z = \hat x \hat w$ for some $w \in \BA_{B_i \setminus \tPhi^h}$ such that $w_{F^j(f_i^\flat)} = y_{f_i^\flat}^{q^j} + P_j(x)$ for $0 \le j \le N-1$, where each $P_j$ is a polynomial function on $\BA_{A_i \setminus B_i}$. Moreover, $P_j = 0$ if $i = m$.

Indeed, the first claim follows from the Proposition \ref{fiber}. Moreover, if $i = m$, then $A_i \setminus B_i \subseteq \tPhi_M$ and $L(\hat x) \in \BM_r^+$. Hence $P_j = 0$ for $1 \le j \le N-1$ by Proposition \ref{fiber}. So (a) is proved.

Then we claim that 

(b) $x_{F^j(f_i)} = x_{f_i}^{q^j}$ for $1 \le i \le m-1$ and $0 \le j \le N-1$.

Indeed, Let $v = x_{\CO_{f_i}} \in \BA_{\CO_{f_i}}$. As $\hat x \in Y^{A_i}_{B_i}$, we have $\hat v \in Y^{A_i}_{A_{i+1}} \subseteq \BG_r^{A_i} / \BG_r^{A_{i+1}} \cong \BA_{\CO_{f_i}}$, that is, $L(\hat v) \in \BA_{f_i} \subseteq \BA_{\CO_{f_i}}$. Now (b) follows from Lemma \ref{zero}.

Assume that $1 \le i \le m-1$. By (b) we have 

(c) $L(\hat x)_f = 0 $ if $f \in \CO_{f_i} \setminus \{f_i\}$ and $L(\hat x)_f = x_{f_i}^{q^N} - x_{f_i}$ if $f = f_i$.

Note that $\hat w, F(\hat w) \in \BG_r^{B_i}$, $\CO_{f_i} < \CO_{f_i^\flat}$ and $B_i + B_i \subseteq \tPhi^{h+}$. Moreover, $L(\hat x) \in \BG_r^{A_i}$ and $[\hat w, (L(\hat x)_{<f_i})\i] \in [\BG_r^{B_i}, \BG_r^{A_{i+1}}] \subseteq \BG_r^{C_{i-1}}$. It follows from Lemma \ref{pr1} and Lemma \ref{pr2} that $(\hat w\i)_h = 0$ and \begin{gather*} (L(\hat x)_{\ge f_i} F(\hat w))_h =  (L(\hat x)_{\ge f_i})_h + F(\hat w)_h = L(\hat x)_h; \\ (\hat w\i [\hat w, (L(\hat x)_{< f_i})\i])_h = (\hat w\i)_h + ([\hat w, (L(\hat x)_{< f_i})\i])_h = 0.\end{gather*} Then one computes that \begin{align*}
    \pi(\psi_i\i(\hat x, y)) &= (\hat w\i L(\hat x) F(\hat w))_h \\ &=(L(\hat x)_{<f_i} \hat w\i [\hat w, (L(\hat x)_{<f_i})\i] L(\hat x)_{\ge f_i} F(\hat w))_h \\ &= (\hat w\i [\hat w, (L(\hat x)_{<f_i})\i] L(\hat x)_{\ge f_i} F(\hat w))_h \\ 
    &= \sum_{f \in \CO_{f_i}} ((\hat w\i)_{h-f} L(\hat x)_f)_h + (\hat w\i [\hat w, (L(\hat x)_{< f_i})\i])_h +  (L(\hat x)_{\ge f_i} F(\hat w))_h \\  &= ((\hat w\i)_{h-f_i} L(\hat x)_{f_i})_h + L(\hat x)_h \\ &= \a_{f_i}^\vee \otimes ((x_{f_i}^{q^N} - x_{f_i}) (y_{f_i^\flat}^{q^{n_i}} + P_{n_i}(x))) + L(\hat x)_h \\ &= \a_{f_i}^\vee \otimes (x_{f_i}^{q^N} - x_{f_i}) y_{f_i^\flat}^{q^{n_i}} + \pi(\psi_i\i(\hat x, 0),
    \end{align*} where the third equality follows from that $f_i < B_i$, the fourth equality follows from Lemma \ref{pr1}, and the fifth equality follows from (c).

    Assume $i = m$. Then $\hat x, L(\hat x) \in \BM_r^+$ and $F^{N/2}(f_i) = h - f_i$. Moreover, by the second statement of (a) we have \begin{align*} \hat w_{\CO_{f_i}}\i F(\hat w_{\CO_{f_i}}) &= (\hat w_{f_i} \cdots F^{N-1}(\hat w_{f_i}))\i F(\hat w_{f_i}) \cdots F^{N-1}(\hat w_{f_i}) \\ &\equiv  \hat w_{f_i}\i F^N(\hat w_{f_i}) [\hat w_{f_i}, F^{N/2}(\hat w_{f_i})\i] \mod \BG_r^{h+}. \end{align*} Thus \[(\hat w\i F(\hat w))_h = (\hat w_{\CO_{f_i}}\i F(\hat w_{\CO_{f_i}}))_h = [\hat w_{f_i}, F^{N/2}(\hat w_{f_i})\i]_h = -\a_{f_i}^\vee \otimes y_{f_i}^{q^{N/2} + 1}.\] As $L(\hat x) \in \BM_r^+$, we have $[\hat w, L(\hat x)\i] \in \BG_r^{C_{i-1}}$. In particular, $[\hat w, L(\hat x)\i]_h = 0$ and $[[\hat w, L(\hat x)\i], F(\hat w)] \in \BG_r^{h+}$. Now we have \begin{align*}
    \pi(\psi_i\i(\hat x, y)) &= (\hat w\i L(\hat x) F(\hat w))_h \\ &= (L(\hat x) \hat w\i [\hat w, L(\hat x)\i] F(\hat w))_h \\ &= (\hat w\i [\hat w, L(\hat x)\i] F(\hat w))_h + L(\hat x)_h \\ &= (\hat w\i F(\hat w) [\hat w, L(\hat x)\i])_h + L(\hat x)_h \\ &= (\hat w\i F(\hat w))_h + [\hat w, L(\hat x)\i]_h + L(\hat x)_h \\
    &= -\a_i^\vee \otimes y_{f_j}^{q^{N/2} + 1} + L(\hat x)_h \\ &= -\a_j^\vee \otimes y_{f_j}^{q^{N/2} + 1} + \pi(\psi_i\i(\hat x, 0)),
\end{align*} where the third (resp. the fifth) equality follows from Lemma \ref{pr2} (resp. Lemma \ref{pr1}). The proof is finished.
\end{proof}

Recall that $\CL_{\chi_{h+}^h}$ is the multiplicative local system on $V = X_*(T) \otimes \overline \BF_q$ corresponding to the character $\chi_{h+}^h: V^F \to \overline \BQ_\ell^\times$.
\begin{proposition} \label{computation}
     Let $\a \in \Phi \setminus \Phi_M$ and let $\k: \BG_a \to V$ be the map given by $x \mapsto \a^\vee \otimes x$ for $x \in \overline \BF_q$.

     (1) $\k^* \CL_{\chi_{h+}^h}$ is nontrivial and hence $H_c^i(\BG_a, \k^* \CL_{\chi_{h+}^h}) = 0$ for $i \in \BZ$;

     (2) if $N$ is even and $F^{N/2}(\a) = -\a$, then \begin{align*} \dim H_c^i(\BG_a, \tau^* \CL_{\chi_{h+}^h}) = \begin{cases} q^{N/2} &\text{ if } i = 1; \\ 0, &\text{otherwise,} \end{cases} \end{align*} where $\tau: \BG_a \to V$ is given by $x \mapsto \a^\vee \otimes x^{q^{N/2} + 1}$. Moreover, in this case $F^N$ acts on $H_c^1(\BG_a, \tau^* \CL_{\chi_{h+}^h})$ via $-q^{N/2}$. 
\end{proposition}
\begin{proof}
Consider the composition of maps \[\th: \BF_{q^N} \overset \k \to V^{F^N} \overset {\Nm_N} \to V^F \overset {\chi_{h+}^h} \to \overline \BQ_\ell^\times,\] where $\Nm_N: V \to V$ is given by $v \mapsto v + \cdots + F^{N-1}(v)$. As $\k$ is a homomorphism defined over $\BF_{q^N}$, we have $\k^* \CL_{\chi_{h+}^h} = \CL_\th$ by Lemma \ref{base-change}. Moreover, since $\a \in \Phi \setminus \Phi_M$, $\th$ is nontrivial by definition. Hence $\CL_\th$ is nontrivial and the statement (1) follows from \cite[Lemma 9.4]{Boyarchenko_10}.

Assume that $N$ is even and $F^{N/2}(\a) = -\a$. Then for $x \in \BF_{q^{N/2}}$ we have \begin{align*} \Nm_N(\a^\vee \otimes x) &= \sum_{i=0}^{N-1} F^i(\a^\vee) \otimes x^{q^i} \\ &= \sum_{i=0}^{N/2 - 1}(F^i(\a^\vee)\otimes x^{q^i} + F^{N/2 + i}(\a) \otimes x^{q^{N/2 + i}}) \\ &= \sum_{i=0}^{N/2 - 1}(F^i(\a^\vee)\otimes x^{q^i} + F^i(-\a^\vee) \otimes x^{q^i}) \\ &=0.\end{align*} Hence the (nontrivial) character $\th$ of $\BF_{q^N}$ restricts to a trivial character of $\BF_{q^{N/2}}$. Now the statement (2) follows from \cite[Proposition 6.6.1]{BoyarchenkoW_16}.
\end{proof}

Let $Z$ be a locally closed subvariety of $\BG_h^+$ with the natural embedding map $i_Z: Z \hookrightarrow \BG_h^+$. For a local system $\CF$ on $\BG_h^+$, we write $H_c^j(Z, \CF) = H_c^j(Z, i_Z^*\CF)$ for simplicity. We set $\pi = \pi_h^{\tPhi^+:\tPhi^h}: \BG_h^+ = \BG_r^+/\BG_r^h \to \BA_h \cong V$.
\begin{proposition} \label{key}
      The following statements hold:
   
         (1) $H_c^j(Y^{A_i}_h, \pi^* \CL_{\chi_{h+}^h}) \cong H_c^j(Y^{A_{i+1}}_h, \pi^* \CL_{\chi_{h+}^h})^{\oplus q^N}$ for $1 \le j \le m-1$;

         (2) $H_c^j(Y^{A_m}_h, \pi^* \CL_{\chi_{h+}^h}) \cong H_c^{j-n-m+1}(Y^M_h, \pi_M^* \CL_{\chi_{h+}^h})^{\oplus q^{(n-m+1)N/2}} ((-q^{N/2})^{n+m-1})$.

     Here $Y^M_h = Y_h \cap \BM_h^+$, and $\pi_M$ is the restriction of $\pi$ to $\BM_h^+$.
\end{proposition}
\begin{proof}
By Proposition \ref{factor} we have an isomorphism \[ \psi_i: Y^{A_i}_h \cong Y^{A_i}_{B_i} \times \BA_{D_i^\flat}.\] Let $p: Y^{A_i}_h \to Y^{A_i}_{B_i}$ be the natural projection. Set $\CL = \CL_{\chi_{h+}^h}$.

Assume $1 \le i \le m-1$. Let $Y^i = \{\hat x \in Y^{A_i}_h; x_{f_i}^{q^N} - x_{f_i} = 0\}$. Then $\psi_i$ restricts to an isomorphism \[Y^i \cong Y_i \times \BA_{D_i^\flat},\] where $Y_i = \{\hat x \in Y^{A_i}_{B_i}; x_{f_i}^{q^N} - x_{f_i} = 0\}$. In view of Proposition \ref{factor} (1), the restriction of $\pi$ to $Y^{A_i}_{B_i} \times \BA_{D_i^\flat}$ is given by \[\pi(\psi_i\i(\hat x, y)) = \pi_i(\psi_i\i(\hat x, y)) = \eta(\hat x, y) + \pi_0(\hat x),\] where $\eta(\hat x, y) = \a_{f_i}^\vee \otimes (x_{f_i}^{q^N}- x_{f_i}) y_{f_i}^{q^{n_i}}$ with $1 \le n_i \le N-1$ such that $F^{n_i}(f_i^\flat) = h - f_i$, and $\pi_0$ is the restriction of $\pi$ to $Y^{A_i}_{B_i} \times \{0\} \subseteq  Y^{A_i}_{B_i} \times \BA_{D_i^\flat}$. As $\CL$ is a multiplicative local system, we have $\pi^* \CL \cong \eta^* \CL \otimes p^* \pi_0^* \CL$. Hence by the projection formula, \[p_! \pi^* \CL  \cong p_! \eta^* \CL \otimes \pi_0^* \CL.\] For $\hat x \in Y^{A_i}_h$ we define $\eta_{\hat x}: \BA_{D_i^\flat} \to V$ be the homomorphism given by $\eta_{\hat x}(y) = \eta(\hat x, y)$. As $\a_{f_i} \in \Phi \setminus \Phi_M$, it follows by Proposition \ref{computation} that $\eta_{\hat x}^* \CL$ is a trivial multiplicative local system if and only if $x_{f_i}^{q^N} - x_{f_i} =0$, that is, $\hat x \in Y_i$. Thus $p_! \pi^* \CL$ is supported on $Y_i \times \BA_{D_i^\flat} \cong Y^i$ and hence $p_! \pi^* \CL \cong p_! (\pi|_{Y^i})^* \CL$. Noticing that \[Y^i = \sqcup_{g \in (\BG_h^{A_i})^F / (\BG_h^{A_{i+1}})^F} g Y^{A_{i+1}}_h\] and that $\#((\BG_h^{A_i})^F / (\BG_h^{A_{i+1}})^F) = \#(\BG_h^{A_i} / \BG_h^{A_{i+1}})^F = q^N$, we have \[H_c^j(Y^{A_i}_h, \pi^* \CL) \cong H_c^j(Y^i, \pi^* \CL) \cong H_c^j(Y^{A_{i+1}}_h, \pi^* \CL)^{\oplus q^N},\] and the first statement is proved.

By Proposition \ref{factor} (2), for $(\hat x, y) \in Y^{A_m}_{B_m} \times \BA_{D_m^\flat} =Y^M_h \times \BA_{D_m^\flat}$ we have \[\pi(\hat x, y) = \tau(y) + \pi_M(\hat x),\] where $\tau(y) = \sum_{j=m}^n \a_{f_j}^\vee \otimes y_{f_j}^{q^{N/2}+1}$. Thus $\pi^* \CL \cong \pi_M^* \CL \boxtimes \tau^* \CL$. By K\"{u}nneth formula, we have \begin{align*}
    &\quad\ H_c^j(Y^{A_m}_h, \pi^* \CL) \\ &\cong \oplus_s H_c^s(\BA_{D_m^\flat}, \tau^* \CL) \otimes H_c^{j-s}(Y^M_h, \pi_M^* \CL) \\ &\cong \otimes_{i=m}^n H_c^1(\BA_{f_i}, \tau_i^* \CL) \otimes \otimes_{i=1}^{m-1} H_c^2(\BA_{f_i^\flat}, \overline \BQ_\ell) \otimes H_c^{j-n-m+1}(Y^M_h, \pi_M^* \CL) \\ &\cong H_c^{j-n-m+1}(Y^M_h, \pi_M^* \CL)^{q^{(n-m+1)N/2}}((-q^{N/2})^{m+n-1}),
\end{align*} where $\tau_i : \BG_a \cong \BA_{f_i} \to V$ is given by $x \mapsto \a_{f_i}^\vee \otimes x^{q^{N/2}+1}$, and the last isomorphism follows from  Proposition \ref{computation} (2). This finishes the proof of the second statement.
\end{proof}

\subsection{Proof of Theorem \ref{main}}\label{sec:proof_main}

Let $G'$, $T'$, $Y_f'$ be as in \S \ref{subsec: derived}. Let $\chi'$ be the restriction of $\chi$ to $(\BT^{\prime +}_f)^F$. By Lemma \ref{der} we have \[H_c^j(Y_f, \overline \BQ_\ell)[\chi] \cong (\ind_{(\BT^{\prime +}_f)^F}^{(\BT_f^+)^F} (H_c^j(Y_f', \overline \BQ_\ell)[\chi']))[\chi].\] So it suffices to prove the theorem for semisimple reductive groups $G = G'$. 

We argue by induction on $f \in \tD^+$ and $\# \Phi$. Indeed, if $f = \min \tD^+$, then $(\BT_f^+)^F = Y_f = \{1\}$ and the statement is trivial. On the other hand, if $\Phi$ is empty, that is, $G = T$, then $Y_f = (\BT_f^+)^F$ is a finite set and the statement is also true. Now we assume the theorem holds for all reductive groups $L$ such that $\# \Phi_L < \# \Phi_G$, and for all $Y_{f'}$ with $f' \le f \in \tD^+$.

If $f \in \D_\aff^+$, by Corollary \ref{coh} (1) we have a $(\BT_f^+)^F$-equivariant isomorphism \[H_c^i(Y_{f+}, \overline \BQ_\ell) = H_c^{i-2}(Y_f, \overline \BQ_\ell)(-q^N).\] Then the statement follows by induction hypothesis.

Now we assume $f = h \in \BZ_{\ge 1}$. Let notation be as in \S\ref{subsec:red}. By Corollary \ref{coh} (2), \[H_c^i(Y_{h+}, \overline \BQ_\ell)[\chi] \cong H_c^i(Y_h, \pi^* \CL_{\chi_{h+}^h})[\chi].\] If $\chi_{h+}^h$ is trivial, then $\CL_{\chi_{h+}^h} = \overline \BQ_\ell$ and $H_c^i(Y_{h+}, \overline \BQ_\ell)[\chi] \cong H_c^i(Y_h, \overline \BQ_\ell)[\chi]$. Hence the statement also follows by induction hypothesis. Assume $\chi_{h+}^h$ is nontrivial and let notation be as in \S \ref{subsec:red}. By Condition \ref{hypo} and Lemma \ref{surj}, we have $M = M_\chi \neq G$. By Proposition \ref{key} and Corollary \ref{coh} (2) we have \begin{align*} &\quad\ H_c^i(Y_h, \pi^* \CL_{\chi_{h+}^h}) [\chi] = H_c^i(Y^{A_1}_h, \pi^* \CL_{\chi_{h+}^h})[\chi]  \\ &\cong (H_c^{i-m-n+1}(Y^M_h, \pi_M^* \CL_{\chi_{h+}^h})[\chi])^{\oplus q^{(m+n-1)N/2}}((-q)^{(m+n-1)N/2}) \\ &\cong (H_c^{i-m-n+1}(Y_{h+}^M, \overline \BQ_\ell) [\chi])^{\oplus q^{(m+n-1)N/2}}((-q)^{(m+n-1)N/2}), \end{align*} So the statement follows by induction hypothesis. The proof is finished.

\subsection{Computation of cohomological degree}\label{sec:computation_coho_degree}
Let $\chi$ be a smooth character $\CT^+(\CO_k)$, which factors through $(\BT_r^+)^F$. We have the Howe factorization of an arbitrary lift of $\chi$ to a smooth character of $T(k)$ from \cite[\S3.6]{Kaletha_19}. We use notation from \emph{loc. cit.} In particular, we have the integers $(r \geq) r_d \geq r_{d-1} > r_{d-2} > \dots > r_0 > 0$ at which the breaks happen and the increasing subsets $R_i:= R_{r_i} \subseteq \Phi$ (which are the roots systems of the twisted Levi subgroups $M_\chi$ appearing in \S\ref{sec:proof_main}). Moreover, $r_{-1} = 0$, $R_d = \Phi$ by definition. Let $R_i^\red = R_i \cap \Phi^\red$, where $\Phi^\red$ is as in \S\ref{sec:Drinfeld_strat}.

\begin{proposition}\label{prop:computation_coho_degree}
We have 
\[
N s_{\chi,r} = 2r\cdot\#\Phi - \#\Phi^\red - \#R_0^\red - \sum_{i=0}^{d-1} r_i(\#R_{i+1} - \#R_i).
\]
\end{proposition}
\begin{proof} 
We can argue by induction on $\# \Phi$ (or on the number of jumps $d$). If $\Phi = \varnothing$, the statement is trivial. Suppose it is true for all reductive groups $L$ with $\#\Phi_L < \#\Phi$. Then in view of \S\ref{sec:proof_main} (where we can assume that $\chi$ is trivial over $\BT_r^{h+1}$ with $h = r_{d-1}$), we have \[s_{\chi,r} = 2(r-r_{d-1}) \cdot \#\D +(m+n-1) + s_{\chi, r_{d-1}}^M,\] where $s_{\chi,r_{d-1}}^M$ is the unique integer $i$ such that $H_c^i(Y_{r_{d-1}}^M, \overline \BQ_\ell)[\chi] \neq 0$. Now, 
\begin{align*} 
m+n-1 &= \#D \\ 
&= \#\{ f \in \Delta_\aff \colon f({\bf x}) > 0, f < r_{d-1} \} \cap (\widetilde R_d \sm \widetilde R_{d-1})) \\ 
&= \frac{1}{N} \left( r_{d-1}  (\#R_d - \#R_{d-1}) - (\#R_d^\red - \# R_{d-1}^\red) \right),
\end{align*}
where $ \widetilde R_{d-1} \subseteq \widetilde\Phi$ is the preimage of $R_i$ under the natural projection $\tPhi \to \Phi \sqcup \{0\}$. Note that $N \cdot \#\D = \#\Phi =\#R_d$. The statement now follows by induction hypothesis. \qedhere
\end{proof}

This generalizes the formula from \cite[Theorem 6.1.1]{CI_DrinfeldStrat}

\begin{corollary}\label{cor:formal_degree}
For the integer $s_\chi$ from Theorem \ref{thm:intro} we have
\[
Ns_\chi = -\#\Phi^\red + \#R_0^\red + \sum_{i=0}^{d-1} r_i(\#R_{i+1} - \#R_i).
\]
\end{corollary}
\begin{proof}
By Lemma \ref{lm:properties_of_Yr}, $2N\dim Y_r = 2N(r \cdot \#\Delta - \#\Delta^\red) = 2(r \cdot \#\Phi - \#\Phi^\red)$. As $Ns_\chi = 2N\dim Y_r - Ns_{\chi,r}$, the claim follows.
\end{proof}

Note that when $\chi$ is sufficiently generic, this, combined with Corollary \ref{cor:dimension}, gives a formula for the formal degree of the corresponding supercuspidal representation. Moreover, note that the essential parts of the formulas of Corollary \ref{cor:formal_degree} and of \cite[Theorem A]{Schwein_24} agree.

\section{Traces}\label{sec:traces}

We combine Theorem \ref{main} with \cite[Lemma 2.12]{Boyarchenko_12} to express the traces of all $g \in \CG(\CO_k)$ on $H_{s_\chi}(Y,\cool)[\chi]$ in terms of the geometry of $Y_h$. In particular, we determine the dimension of $H_{s_\chi}(Y,\cool)[\chi]$ in terms of the non-vanishing degree $s_{\chi}$.

\begin{proposition}\label{cor:traces}
Let $\chi \colon \CT^+(\CO_k) \rar \cool^\times$ be a smooth character which factors through $(\BT_r^+)^F$. Let $g \in \CG^+(\CO_k)$ with image $\bar g \in (\BG_r^+)^F$. Then 
\[ {\rm tr}(\bar g, H_c^{s_{\chi,r}}(Y_r,\cool)[\chi]) = \frac{1}{\#(\BT_r^+)^F \cdot q^{s_{\chi,r} N/2}} \sum_{t \in (\BT_r^+)^F} \chi(t)\cdot \#S_{g,t}, \]
where $S_{g,t} = \{x \in Y_r(\obF) \colon gF^N(x) = x t\}$.
For $g=1$ this simplifies to
\[ \dim_{\cool} H_c^{s_{\chi,r}}(Y_r,\cool)[\chi] = \frac{\#(\BG_r^+)^F}{\#(\BT_r^+)^F \cdot q^{s_{\chi,r} N/2}}. \]
\end{proposition}

\begin{proof}
The first statement follows from Theorem \ref{main} by \cite[Lemma 2.12]{Boyarchenko_12}. For the second statement we have to compute the trace for $g=1$. Therefore, let $x \in S_{1,t}$ for some $t \in (\BT_h^+)^F$ and put $u = x^{-1}F(x)$. Then $x \in S_{1,t}$ implies $t = x^{-1}F^N(x) = \prod_{i=0}^{N-1} F^i(u)$. We claim that this implies $t=u=1$. 
Let $A \subseteq \widetilde\Phi^+$ be an $F$-stable and closed subset. Suppose that we have already shown that $t,u \in \BG_r^A$. Let $f \in A$ be such that $f({\bf x})$ is minimal among all roots in $A$. Then $A \sm \CO_f \subseteq A$ is $F$-stable and closed, and $A + A \subseteq A \sm \CO_f$, so that $\BG_h^{A\sm \CO_f} \subseteq \BG_h^A$ is normal with abelian quotient. By induction on $A$ it suffices to show that $t,u \in \BG_h^{A \sm \CO_f}$. Let $\bar t,\bar u \in \BG_h^A / \BG_h^{A \sm \CO_f}$ denote the images of $t,u$. If $f \in \BZ_{>0}$, then $\bar u = 1$ and hence also $\bar t = 1$, so that we are done. If $f \not\in \BZ_{>0}$, then $t = 1$ and $\BG_h^A / \BG_h^{A \sm \CO_f} \cong \prod_{i=0}^{N-1} \BG_a$, with $F$-action given by $F((x_i)_{i=0}^{N-1}) = (x_{i-1}^q)_{i=0}^{N-1}$ (the $i$th copy of $\BG_a$ corresponds to $F^i(f)$). Now, as $u \in \overline\BU_h \cap F\BU_h^-$ by assumption, $\bar u$ corresponds under this isomorphism to an element of the form $(a,0, \dots, 0)$ with $a\in \BG_a$, and the equation $\prod_{i=0}^{N-1} F^i(\bar u) = 1$ in $\BG_h^A / \BG_h^{A \sm \CO_f}$ thus corresponds to $(a,a^q, \dots, a^{q^{N-1}}) = 0$. Thus $a = 0$, i.e., $\bar u = 1$ and our original claim follows by induction on $A$. The claim immediately implies $S_{1,t} = \varnothing$ unless $t=1$ and $S_{1,1} = (\BG_h^+)^F$ which proves the proposition.
\end{proof}

\begin{proof}[Proof of Corollary \ref{cor:dimension}]

Let $r \in \BZ_{\ge 1}$ such that $\chi$ factors through $\BT_r^+$. It follows from \S\ref{sec:homology_in_general} that $s_{\chi,r} = 2\dim(Y_r) - s_\chi = 2\dim (\overline\BU_r^+ \cap F\BU_r^+) - s_\chi$. Note that \[N \dim (\overline\BU_r^+ \cap F\BU_r^+) = \# (\tPhi_r^+ \cap \tPhi_\aff) = \dim \BG_r^+ - \dim \BT_r^+.\] Thus
\begin{align*}
q^{s_{\chi,r}N/2} &= q^{N\dim (\overline\BU_r^+ \cap F\BU_r^+) - \frac{s_\chi N}{2}} = q^{\dim (\BG_r^+ / \BT_r^+) - \frac{s_\chi N}{2}} \\ &= \frac{\#(\BG_r^+)^F}{\#(\BT_r^+)^F} \cdot q^{-s_\chi N/2}.
\end{align*}
Inserting this into the second formula of Proposition \ref{cor:traces} gives the result.
\end{proof}

\begin{corollary} \label{max}
    Assume that $p$ satisfies Condition \ref{hypo}. The varieties $Y_f$ for $f \in \tPhi_r^+$ are maximal. In particular, the varieties $X_r^{(\BT_{0+})}$ for $r \in \BZ_{\ge 0}$ are maximal.
\end{corollary}
\begin{proof}
    By definition we need to show that either $H_c^s(Y_f, \overline \BQ_\ell) = 0$ or $F^N$ acts on $H_c^s(Y_f, \overline \BQ_\ell)$ by the scalar $(-1)^s q^{sN/2}$ with $sN$ even. By Proposition \ref{fiber}, we can replace $f$ with $h = \min\{n \in \BZ: n \ge f\}$. 

    Assume that $H_c^s(Y_h, \overline \BQ_\ell) \neq 0$. Then there exists a character $\chi$ of $\BT_h^+$ such that $H_c^s(Y_h, \overline \BQ_\ell)[\chi] \neq 0$. By Theorem \ref{main}, $s = s_{h, \chi}$ and $F^N$ acts on $H_c^s(Y_h, \overline \BQ_\ell)[\chi]$ by the scalar $(-1)^s q^{sN/2}$. It remains to show $sN$ is even. In view of Proposition \ref{prop:computation_coho_degree}, this question is combinatorial and we may assume that $q$ is a suitable prime number. Then it follows from Corollary \ref{cor:dimension} that $s_\chi N$ is even. As $s = 2\dim Y_h - s_\chi$, we deduce that $s N$ is even as desired. \qedhere
    
\end{proof}

\section{Irreducibility}\label{sec:irreducibility}
\emph{Until the end of this article we assume that $(T, U)$ is a Coxeter pair.} 

\smallskip

Recall the minimal Drinfeld stratum $X^{(\BT_{0+})}$ of $X\subseteq \BG$ from \S\ref{sec:Drinfeld_strat}. We have its subscheme $Y$ and the slightly bigger subscheme  
\[ 
Z = X^{(\BT_{0+})} \cap \BT \BG^+ = \{g \in \BT \BG^+ \colon g\i F(g) \in \overline\BU \cap F\BU \}
\] 
We have the corresponding approximations $Y_r \subseteq \BG_r^+$ and $Z_r \subseteq \BT_r \BG_r^+$; $Y_r$ is equipped with an $(\BG_r^+)^F \times (\BT_r^+)^F$-action and $Z_r$ is equipped with an $(\BT_r\BG_r^+)^F \times \BT_r^F$-action.

In Theorem \ref{main} we have seen that for any $\chi \colon (\BT_r^+)^F \rar \cool^\times$, $H^\ast_c(Y_r)[\chi]$ is concentrated in one degree. By Lemma \ref{lm:decomposition_closed_Drinfeld_stratum}, the same holds also for $Z_r$ for any character $\chi \colon \BT_r^F \rar \cool^\times$. Now we prove that these weight spaces are irreducible as $\BG_r^F$ (resp. $(\BG_r^+)^F$-)representations and pairwise distinct.

\begin{theorem}\label{thm:irred}
For any $\chi,\chi' \colon \BT_r^F \rar \cool^\times$ we have
\[
\left\langle H^\ast_c(Z_r)[\chi],H^\ast_c(Z_r)[\chi'] \right\rangle_{(\BT_r \BG_r^+)^F} = \begin{cases} 1 & \text{if $\chi = \chi'$,} \\ 0 & \text{otherwise.} \end{cases}
\]
The same holds for $Y_r$, when $(\BT_r)^F$, $(\BT_r \BG_r^+)^F$ are replaced by $(\BT_r^+)^F$, $(\BG_r^+)^F$.
\end{theorem}
\begin{proof}
Let 
\[\Sigma = \{(y, x, x') \in \BT_r \BG_r^+ \times (\overline\BU_r^+ \cap F\BU_r^+) \times (\overline\BU_r^+ \cap F\BU_r^+); y\i x F(y) = x'\},\]
equipped with $\BT_r^F\times \BT_r^F$-action by $(t,t') \colon (y,x,x') \mapsto (tyt^{\prime -1},txt\i,t'x't^{\prime -1})$. As in \cite[\S6.6]{DeligneL_76} we have $\Sigma \cong (\BT_r \BG_r^+)^F \backslash (Z_r \times Z_r)$. It thus suffices to show that $H^\ast_c(\Sigma) \cong H_c^*(\BT_r^F)$. 

By Iwahori decomposition we have $y = \t y_+ y_-$ with $y_+ \in \BU_r^+$, $\t \in \BT_r$ and $y_- \in \overline\BU_r^+$. Then the equality $y\i x F(y) = x'$ is equivalent to \[\tag{a} y_+\i \t\i x F(\t) F(y_+) F(y_-) = y_- x'.\] 

By Theorem \ref{St} (2) there is a unique pair $(u_1, u_2) \in (\BU_r^+ \cap F\i\overline\BU_r^+)  \times \BU_r^+$ such that \[\tag{*} y_+\i \t\i x F(\t) F(y_+) = u_2 \t\i F(\t) F(u_1),\] and moreover, the correspondence $(\t, x, y_+) \mapsto (\t, u_1, u_2)$ gives an isomorphism \[\BT_r \times (\overline\BU_r^+ \cap \BU_r^+) \times \BU_r^+ \cong \BT_r \times (\BU_r^+ \cap F\i\overline\BU_r^+) \times \BU_r^+.\] Now the equality (a) becomes \[\tag{b} u_2 \t\i F(\t) F(u_1) F(y_-) = y_- x'.\] 

Write $y_- = y_1 y_2$ with $y_1 \in \overline\BU_r^+ \cap F\i(\BU_r^+)$ and $y_2 \in \overline\BU_r^+ \cap F\i(\overline\BU_r^+)$. By Theorem \ref{St} (1), the map $(x', y_2) \mapsto u_- := y_2 x' F(y_2)\i$ gives an isomorphism $(F\BU_r^+ \cap \overline\BU_r^+) \times (\overline\BU_r^+ \cap F\i\overline\BU_r^+) \cong \overline\BU_r^+$. Thus the equality (b) becomes \[\tag{c} u_2 \t\i F(\t) F(u_1 y_1) = y_1 u_-.\]

Write $u_1 y_1 = z_1 z_0 z_2$ with $z_1 \in F\i(\BU_r^+)$, $z_0 \in \BT_r$ and $z_2 \in F\i\overline\BU_r^+$. Then the equality (c) becomes \[\tag{d} u_2 {}^{\t\i F(\t)} F(z_1) \t\i F(\t) F(z_0) F(z_2) = y_1 u_-.\] It follows from (d) that $u_2 = {}^{\t\i F(\t)} F(z_1)\i$, $\t\i F(\t) = F(z_0)\i$ and $u_- = y_1\i F(z_2)$. Thus we deduce that \[\tag{e} \Sigma \cong \{(\t, u_1, y_1) \in \BT_r \times (\BU_r^+ \cap F\i\overline\BU^+) \times (\overline\BU_r^+ \cap F\i\BU_r^+); \t F(\t)\i = \pr_0(F(u_1 y_1))\},\] where $\pr_0: \BT_r\BG_r^+ \cong \BU_r^+ \times \BT_r \times \overline\BU_r^+ \to \BT_r$ is the natural projection.

Note that $(\z, \xi) \in \BT_r^F \times \BT_r^F$ acts on $\Sigma$ by $(y, x, x')\mapsto (\z y \xi\i, \z x \z\i, \xi x' \xi\i)$. Then $(\z, \xi)$ sends $(\t, x, y_+, y_-)$ to $(\t \z\xi\i, \z x \z\i, \xi y_+ \xi\i, \xi y_- \xi\i)$. Using the relation (*) we see that $(\z, \xi)$ sends $(u_1, u_2)$ to $(\xi u_1 \xi\i, \xi u_2 \xi\i)$. Therefore, in view of (e), $(\z, \xi)$ acts on $\Sigma$ by sending $(\t, u_1, y_1)$ to $(\t \z \xi\i, \xi u_1 \xi\i, \xi y_1 \xi\i)$.

Let $\eta \in \BT_r$. Consider the action of $\eta$ on $\Sigma$ by sending $(\t, u_1, y_1)$ to $(\t, \eta u_1 \eta\i, \eta y_1 \eta\i)$. Then the action of $\BT_r$ commutes with the action of $\BT_r^F \times \BT_r^F$. Thus, we have an $\BT_r^F \times \BT_r^F$-equivariant isomorphism \[H_c^*(\Sigma) \cong H_c^*(\Sigma^{\BT_{r, {\rm red}}}) \cong H_c^*(\BT_r^F),\]
where $\BT_{r, {\rm red}}$ denotes the reductive part of $\BT_r$. Now the statement follows. The proof for $Y_r$ is the same.
\end{proof}

\section{Relation to the orbit method}\label{sec:Kirillow}

Let $r \in \BZ_{\geq 1} \cup \{\infty\}$. We have the groups $\BG_r^+$ and $\BT_r^+$ and the variety $Y_r$ with $(\BG_r^+)^F \times (\BT_r^+)^F$-action (where we put $Y_\infty=Y$, $\BG_\infty^+ = \prolim_r \BG_r^+$ and similarly for $\BT_\infty^+$). Theorems \ref{main} and \ref{thm:irred} show that $H_c^{s_{\chi,r}}(Y_r,\cool)[\chi]$ is an irreducible $(\BG_r^+)^F$-representation. On the other hand, if either $r<p$, or $r = \infty$ and $(\BG_r^+)^F$ is uniform (see below), Kirillov's orbit method attaches irreducible $(\BG_r^+)^F$-representations to adjoint $(\BG_r^+)^F$-orbits in the dual of the Lie algebra of $(\BG_r^+)^F$. We state a conjecture about the relation between these two constructions and verify it in a non-trivial case.

\subsection{Review of the orbit method}\label{sec:review_orbit}
The orbit method was originally developed by Kirillov \cite{Kirillov_62} and extended later to various related setups. We briefly recall it in the two setups relevant for our article. We refer to \cite{BoyarchenkoS_08} (in particular, \S2 therein), \cite[\S2]{BoyarchenkoD_10} and \cite{DDMS_99} and references therein for more detailed discussions. 

Assume that $p>2$.\footnote{This assumption can be weakened at the cost of more technical results.} For the first setup, recall that a \emph{uniform} Lie algebra is a (topological) Lie algebra $\fkg$ over $\BZ_p$, which is free of finite rank as a $\BZ_p$-module and satisfies $[\fkg,\fkg] \subseteq p\fkg$.
Following Lazard, there is a pro-$p$ group $\Gamma = \exp \fkg$ attached to $\fkg$, whose underlying topological space is $\fkg$ and on which the group law is defined (via $\exp$ and $\log$) by the Campbell--Hausdorff series. For $\Gamma = \exp \fkg$, one has mutually inverse homeomorphisms $\exp \colon \Gamma \rar \fkg$ and $\log \colon \fkg \rar \Gamma$. Set up appropriately, the functor $\fkg \mapsto \exp \fkg$ even defines an isomorphism of categories. We denote the inverse functor by $\Gamma \mapsto \log \Gamma$. A profinite group $\Gamma$ is called \emph{uniform} (short for \emph{uniformly powerful}) if there is a uniform Lie-algebra $\fkg$ with $\Gamma \cong \exp \fkg$. There is a similar isomorphism of categories between finite $p$-groups $\Gamma$ of nilpotence class $<p$ and finite nilpotent Lie rings $\fkg$ of $p$-power order and nilpotence class $<p$. We use the same notation as in the uniform pro-$p$ case.

For the moment, let $\Gamma$ be either
\begin{itemize}
\item[(i)] a uniform pro-$p$ group, or
\item[(ii)] a finite $p$-group of nilpotence class $<p$. 
\end{itemize}
Let $\fkg = \log \Gamma$ denote the corresponding uniform Lie $\BZ_p$-algebra resp. finite Lie ring. Let $\widehat \Gamma$ denote the set of isomorphism classes of smooth irreducible $\cool$-representations of $\Gamma$.
Note that there is an adjoint action of $\Gamma$ on $\fkg$. More precisely, for any $g \in \Gamma$ we have the automorphism $\Ad\, g \colon \fkg \rar \fkg$ given by $x \mapsto \log(g\exp(x)g^{-1})$. Let 
\[
\fkg^\ast = \Hom_{\rm cont}(\fkg,\cool^\times). 
\]
be the dual of $\fkg$. The adjoint action of $\Gamma$ on $\fkg$ induces an action of $\Gamma$ on $\fkg^\ast$. Kirillov's orbit method, in the present setup established in \cite{BoyarchenkoS_08}, describes a natural bijection between $\widehat \Gamma$ and the set of $\Gamma$-orbits in $\fkg^\ast$.

\begin{theorem}[Theorem 2.6 in \cite{BoyarchenkoS_08}]\label{thm:orbit_method}
Assume $p\geq 3$ and $\Gamma$ is either a uniform pro-$p$-group or a $p$-group of nilpotence class $<p$ and let $\fkg = {\rm Lie}\,  \Gamma$. Then there exists a bijection $\Omega \leftrightarrow \rho_\Omega$  between $\Gamma$-orbits $\Omega \subseteq \fkg^\ast$ and $\widehat \Gamma$, characterized by
\[ \tr(g,\rho_\Omega) = \frac{1}{\#\Omega^{1/2}} \cdot \sum_{f\in \Omega} f(\log(g)).
\]
\end{theorem}

Groups of the form $\Gamma = (\BG_r^+)^F$ may or may not satisfy the assumptions of Theorem \ref{thm:orbit_method}, as the following examples show.

\begin{example}\label{ex:when_orbit_method_applies_profinite}
Suppose $r = \infty$. Then $\Gamma = (\BG_r^+)^F$ is the maximal pro-$p$ subgroup of the parahoric group $\CG(\CO_k)$. If $\charac k = p$, $\Gamma$ always contains torsion, and hence is never uniform. Suppose now that $\charac k = 0$. Then $\Gamma$ might or might not be uniform. For example, $1 + p M_n(\BZ_p)$ is uniform. On the other hand, if $k/\BQ_p$ has ramification index $e>1$, then $1+\varpi M_n(\CO_k)$ is not uniform. In general, it is true that a topological group has the structure of a $p$-adic Lie group if and only if it contains an open uniform subgroup \cite[Theorems 8.1 and 4.2]{DDMS_99}.
\end{example}

\begin{example}\label{ex:when_orbit_method_applies_finite}
Suppose $r<\infty$. If $\bfx$ is hyperspecial, $\Gamma = (\BG_r^+)^F$ is of nilpotency class $\leq r-1$ (as $f(\bfx)$ is integral for all $f \in \widetilde\Phi$, $\BG_r^+ = \BG_r^1$, and the subgroups $\{\BG_r^i\}_{i=1}^{r}$ form a central series of length $r-1$). Thus if $r \leq p$, the orbit method applies to the finite $p$-group $\Gamma$. In contrast to Example \ref{ex:when_orbit_method_applies_profinite}, there is no assumption on the characteristic of $k$.
\end{example}

\subsection{Cohomological induction vs. the orbit method}\label{sec:coho_vs_orbit}
For brevity we write $\Gamma = (\BG_r^+)^F$ and  $\Upupsilon = (\BT_r^+)^F$. Note that $\Upupsilon$ satisfies condition (i) or (ii) in \S\ref{sec:review_orbit} and let $\fkt = \log \Upupsilon$ denote its Lie algebra. As $\Upupsilon$ is abelian, $\exp_\Upupsilon \colon \fkt \rar \Upupsilon$ is not only a homeomorphism, but also an isomorphism of groups with inverse $\log_\Upupsilon$. Also, as $\Upupsilon$ is abelian, we may identify $\widehat \Upupsilon$ with $\Upupsilon^\ast := \Hom_{\rm cont}(\Upupsilon,\cool^\times)$. By Theorem \ref{thm:intro} we get a map 
\[
R_{\log} \colon \fkt^\ast \stackrel{\log_\Upupsilon^\ast}{\longrightarrow} \Upupsilon^\ast \rar \widehat \Gamma, 
\]
where the second map is 
\[
\chi \mapsto (-1)^{s_\chi} H_{s_\chi}(Y,\cool)[\chi]
\]
On the other hand, Theorem \ref{thm:orbit_method} gives a map
\[
\rho \colon \fkg^\ast \rar \fkg^\ast/\Ad \, \Gamma \stackrel{\sim}{\rar} \widehat \Gamma
\]
where the first arrow is the natural projection. It is a natural question, how these two maps are related, and we make the following conjecture in this direction. Note that there is a \emph{canonical} projection 
\[
\delta \colon \fkg \tar \fkt 
\] 
(as on the level of (geometric points of) the Lie algebras, $\fkt$ is the weight $0$ subspace of the adjoint representation of $\Upupsilon$ on $\fkg$; then one takes Frobenius invariants). Let $\delta^\ast \colon \fkt^\ast \rar \fkg^\ast$ be the dual map.

\begin{conjecture}\label{conj:relation_to_orbit_method}
We have $\rho \circ \delta^\ast = R_{\rm log}$. 
\end{conjecture}

With other words, if $\chi \in \widehat{\Upupsilon}$ is a character, then Conjecture \ref{conj:relation_to_orbit_method} predicts that $H_{s_\chi}(Y_r,\cool)[\chi] \cong \rho_{\Omega}$, where $\Omega \in \fkg^\ast/\Ad\, \Gamma$ is the orbit of $\delta^\ast(\chi \circ \exp_\Upupsilon) = \chi \circ \exp_\Upupsilon \circ \delta \in \fkg^\ast$. Note that to be able to state the conjecture we need (only) Theorem \ref{thm:irred} , but to verify it in a special case in \S\ref{sec:example_orbit} we use Theorem \ref{main}.

\begin{remark}
(1) Combined with \cite[Theorems 2.9 and 2.11]{BoyarchenkoD_10}, Conjecture \ref{conj:relation_to_orbit_method} allows a realization of $H_{s_\chi}(Y_r,\cool)[\chi]$ as an induced representation (at least in the case when $\Gamma$ is finite).
\smallskip

\noindent (2) In the light of Examples \ref{ex:when_orbit_method_applies_profinite} and \ref{ex:when_orbit_method_applies_finite}, Conjecture \ref{conj:relation_to_orbit_method} says that $\chi \mapsto H_{s_\chi}(Y_r,\cool)[\chi]$ is a generalization of the orbit method (for those adjoint orbits containing an element of $\fkt$) to all groups of the form $\Gamma = (\BG_r^+)^F$. The collection of all such groups is neither contained in, nor containing the family of groups for which the orbit method applies.
\end{remark}

\subsection{An example.}\label{sec:example_orbit} Assume that ${\rm char}(k) = p > 2$, let $G = \GL_2$ and $r=3$. Let $\CG$ be the standard model of $G$ over $\CO_k$. We verify Conjecture \ref{conj:relation_to_orbit_method} in this case, that is for the group 
\[
\Gamma = 1 + \varpi M_2(\BF_q[\![\varpi]\!])/1+ \varpi^3 M_2(\BF_q[\![\varpi]\!]),
\]
where $M_2$ denotes the $2\times 2$-matrices. ($\Gamma$ is of nilpotency class $2<p$, hence the orbit method applies.) 

\smallskip

Write $\overline R := \obF[\varpi]/(\varpi^2)$ with Frobenius $\sigma(a + \varpi b) = a^q + \varpi b^q$ and let $R := \overline R^\sigma$ and $R_2 := \overline R^{\sigma^2}$. Write
\[
x(g_1,g_3) := 1 + \varpi \left(\begin{smallmatrix}
  g_1 & \sigma(g_3) \\ g_3 & \sigma(g_1) \end{smallmatrix}\right) \in 1 + \varpi M_2(\overline R) \cong \frac{1 + \varpi M_2(\obF[\![\varpi]\!])}{1+ \varpi^3 M_2(\obF[\![\varpi]\!])}
\]
with $g_i = g_{i0} + \varpi g_{i1} \in \overline R$ for $i=1,3$. 
Let $F = \Ad \begin{psmallmatrix} 0 & 1\\ 1 & 0\end{psmallmatrix} \circ \sigma$ be the twisted Frobenius on $1 + \varpi M_2(\overline R)$, such that the diagonal torus in $\Gamma$ becomes the unramified elliptic torus. We get a presentation of $\Gamma$ as
\[
\Gamma \cong \left(1 + \varpi M_2(\overline R)\right)^F = \left\{ x(g_1,g_3) \colon g_1,g_3 \in R_2 \text{ for $i=1,3$} \right\},
\]
which will be in use until the end of \S\ref{sec:example_orbit}. Then $\Upupsilon = \{g_3 = 0\} \subseteq \Gamma$ and the corresponding deep level Deligne--Lusztig space $Y_3$ is given by
\[
Y_3 = \left\{ x(v_1,v_3) \in 1 + \varpi M_2(\overline R) \colon \det x(v_1,v_3) \in R ^\times \right\}
\]
The condition $\det x(v_1,v_3) \in R ^\times$ is equivalent to the conditions $v_{10} \in \BF_{q^2}$ and $v_{11}^{q^2} - v_{11} = v_{30}^{q^2+q} - v_{30}^{q+1}$. Next we describe how $\Gamma$ and $\Upupsilon$ act on a point $x(v_1,v_3) \in Y_3$. Let $t = x(\tau,0) \in \Upupsilon$ with $\tau = \tau_0 + \varpi \tau_1 \in R_2$. Then for we have 
\[ 
x(v_1,v_3).t = x(v_{10} + \tau_0 + \varpi(v_{11} + \tau_1 + v_{10}\tau_0),v_{30} + \varpi (v_{31} + v_{30}\tau_0)).
\]
Let $g = g(g_1,g_3) \in \Gamma$. Then 
\begin{align*}
g.x(v_1,v_3) = x(v_{10} + g_{10} + &\varpi(v_{11} + g_{11} + g_{10}v_{10} + g_{30}^q v_{30}), \\ &v_{30} + g_{30} + \varpi(v_{31} + g_{31} + g_{30}v_{10} + g_{10}^q v_{30})).
\end{align*}

\begin{lemma}\label{lm:geom_traces}
There exists a constant $C \in \BQ^\times$ such that for all $g = x(g_1,g_3) \in \Gamma$ one has
\[
\tr(g, H_{s_\chi}(Y_3,\cool)[\chi]) = \begin{cases} Cq\cdot \chi(x(g_1, 0)) & \text{if $g_{30} = 0$}, \\
C \cdot\sum\limits_{\lambda \in \BF_{q^2} \colon \lambda^{q}+\lambda = g_{30}^{q+1}} \chi(x(g_{10} + \varpi (g_{11} - \lambda), 0)) & \text{otherwise.}
\end{cases}
\]
\end{lemma}
\begin{proof}
$Y_3$ is defined over $\BF_{q^2}$. Combining Theorem \ref{main} with \cite[Lemma 2.12]{Boyarchenko_12}, we see that there is some $C_1 \in \BQ^\times$, such that for any $g = g(g_1,g_3) \in G$, 
\[
\tr(g,|R_\chi|) = C_1 \cdot \sum_{t \in T} \#S_{g,t} \cdot \chi(t),
\]
where
\[
S_{g,t} = \{x \in X \colon g.F^2(x) = x.t \}.
\]
Write $t = x(\tau,0)$. Using the above description of the actions on $X$, one easily sees that $S_{g,t} = \varnothing$ unless $g_{10} = \tau_0$. Assume that $g_{10} = \tau_0$ holds. Using the determinant condition above, one easily deduces that $\#S_{g,t} = q^6 \cdot \#S_{g,t}'$, where
\[
S_{g,t}' = \{v_{30} \in \BF_{q^2} \colon v_{30} - v_{30}^{q^2} = g_{30}\text{ and } g_{11} - \tau_1 - g_{30}^{q+1}= v_{30}^q g_{30} - v_{30}g_{30}^q \}
\]
If $g_{30} = 0$, the claim of the lemma becomes clear now. Assume $g_{30}\neq 0$. Suppose first that $\tau_1$ is such that $S_{g,t}' \neq \varnothing$. Then, if $v_{30} \in S_{g,t}'$ is arbitrary, writing $y := v_{30}^q g_{30} - v_{30}g_{30}^q$ we see (using that $v_{30}^{q^2} = v_{30}-g_{30}$) that $y^q = -y - g_{30}^{q+1}$. But on the other hand, $g_{11} - \tau_1 = y + g_{30}^{q+1}$, and hence we deduce (using that $g_{30} \in \BF_q$) that
\[
(g_{11} - \tau_1)^q + (g_{11} - \tau_1) = (y + g_{30}^{q+1})^q + (y + g_{30}^{q+1}) = y^q + y + 2g_{30}^{q+1} = g_{30}^{q+1}.
\]
With other words, $S_{g,t}' = \varnothing$, unless 
\begin{equation}\label{eq:aux_eq}
(g_{11} - \tau_1)^q + (g_{11} - \tau_1)  = g_{30}^{q+1}.
\end{equation}
Assume now that this equality holds. Note that $v_{30}^q g_{30} - v_{30}g_{30}^q = g_{11} - \tau_1 - g_{30}^{q+1}$, regarded as an equation in $v_{30}$, has precisely $q$ different solutions in $\obF$ (as $g_{30} \neq 0$). Moreover, if $v_{30}$ is one of its solutions, then (applying the transformation $X \mapsto X^q + X$ to both sides of this equation) one verifies using \eqref{eq:aux_eq} that $v_{30}$ also satisfies $v_{30}^{q^2} - v_{30} = -g_{30}$, that is $v_{30} \in S_{g,t}'$. Altogether, $\#S_{g,t}' = q$ if \eqref{eq:aux_eq} holds and $\#S_{g,t}' = 0$ otherwise. The lemma follows immediately from this by taking $\lambda := g_{11} - \tau_1$ for those $\tau_1$ which satisfy \eqref{eq:aux_eq}.
\end{proof}

Now we consider the orbit method side. Write  $y(g_1,g_3) := x(g_1,g_3) - 1 \in \varpi M_2(\overline R) = \fkg = {\rm Lie}\, \Gamma$. The map $\log \colon \Gamma \rar \fkg$ is given by $\log(1+\varpi z) = \varpi z - \frac{\varpi^2 z}{2}$.
Let 
\[ 
\delta \colon \fkg \rar \fkt, \quad y(g_1,g_3) \mapsto y(g_1, 0) \quad \text{ and let } \quad  \varepsilon := \chi \circ \exp_T \circ \delta \in \fkg^\ast.
\]
Consider first the $\Gamma$-orbit $\Omega_\delta$ of $\delta$ ($\Gamma$ acts on the first factor in $\Hom(\fkg,\fkt)$ by conjugation). First, note that the action of $\Gamma$ factors through $\Gamma = (\BG_3^+)^F \tar (\BG_2^+)^F = 1+ \varpi M_2(\BF_{q^2})$. Moreover, for $h = x(h_{10},h_{30})\in (\BG_2^+)^F$ we have
\begin{align*}
(\Ad h)(\delta)(y(g_1,g_3)) &= \delta(hy(g_1,g_3)h^{-1}) \\ &= y(g_{10} + \varpi(g_{11} + h_{10}^q g_{30} - h_{10} g_{30}^q),0) =: \delta_{h_{10}}(g).
\end{align*}
Thus, $\Omega_\delta = \{\delta_{h_{10}} \colon h_{10} \in \BF_{q^2}\}$ has cardinality $q^2$. As $\exp_\Upupsilon$ is an isomorphism, the $\Gamma$-orbit $\Omega_{\exp_\Upupsilon \circ\delta} = \exp_\Upupsilon \circ\Omega_\delta$ of $\exp_\Upupsilon \circ \delta \in \Hom(\Gamma,\fkt)$ has the same cardinality as $\Omega_\delta$. Let now $h_{10} \neq h_{10}' \in \BF_{q^2}$. An easy computation shows that $\chi\circ\exp_\Upupsilon \circ\delta_{h_{10}} = \chi\circ\exp_\Upupsilon\circ\delta_{h_{10}'}$ if and only if 
$\chi|_{1 + \varpi^2\BF_q^{-}}$ is trivial, where we set $\BF_q^- := \{x \in \BF_{q^2} \colon x+x^q = 0\}$. 

\smallskip

Suppose first that $\chi|_{1 + \varpi^2\BF_q^{-}}$ non-trivial. Then composition with $\chi\circ\exp_\Upupsilon$ induces a bijection $\Omega_\delta \stackrel{\sim}{\rar} \Omega_{\varepsilon}$. Unraveling the trace formula from Theorem \ref{thm:orbit_method} we then that for $g = x(g_1,g_3)$:
\begin{equation}\label{eq:trace_orbit}
\tr(g,\rho_{\Omega_{\varepsilon}}) = C_2 \cdot \sum_{\alpha \in \BF_{q^2}}\chi(x(g_{10} + \varpi(g_{11} - \frac{g_{30}^{q+1}}{2} + \alpha^q g_{30} - \alpha g_{30}^q))),
\end{equation}
for some constant $C_2 \in \BQ^\times$. If $g_{30} = 0$, this clearly agrees with the trace from Lemma \ref{lm:geom_traces} up to a (non-zero) scalar. Assume $g_{30} \neq 0$. Then the homomorphism $\alpha\mapsto \alpha^q g_{30} - \alpha g_{30}^q \colon \BF_{q^2} \rar \BF_{q^2}$ is easily seen to have image $\BF_q^-$. Thus, \eqref{eq:trace_orbit} transforms to
\[
\tr(g,\rho_{\Omega_{\varepsilon}}) = C_2 \cdot q \sum_{\mu \in \BF_q^-}\chi(x(g_{10}+\varpi(g_{11} - \frac{g_{30}^{q+1}}{2} + \mu)))
\]
Now it is immediate to check that the map $\mu \mapsto \lambda := \frac{g_{30}^{q+1}}{2} - \mu$ defines a bijection between $\BF_q^-$ and the set $\{\lambda \in \obF \colon \lambda^q + \lambda = g_{30}^{q+1}\}$. Thus the trace of $\rho_{\Omega_\varepsilon}$ agrees with the trace from Lemma \ref{lm:geom_traces} up to a non-zero scalar, which does not depend on $g$. As we know that $H_{s_\chi}(Y_3,\cool)[\chi]$ and $\rho_{\Omega_\varepsilon}$ are both irreducible $\Gamma$-representations, it follows that they must be isomorphic. 

\smallskip

In the case that $\chi|_{1 + \varpi^2\BF_q^{-}}$ is trivial, a similar (and easier) computation leads to the same conclusion. Altogether we have shown:

\begin{proposition}
For $\varepsilon = \chi\circ\exp_\Upupsilon\circ\delta$ we have $H_{s_\chi}(Y_3,\cool)[\chi] \cong \rho_{\Omega_{\varepsilon}}$ as $\Gamma$-representations. Thus, Conjecture \ref{conj:relation_to_orbit_method} holds for $\Gamma$. 
\end{proposition}

\appendix

\section{Algorithm for the Steinberg cross-section}\label{sec:appendix_sage}

The algorithm used in the proof of Proposition \ref{St} consists of two procedures (implemented in SAGE, v8.6), which we now describe.

\smallskip

\noindent \emph{find\_candidate\_for\_one\_step} (procedure 1):

\noindent \textbf{Input}: an element $w \in W$, a set $\Psi \subsetneq \Phi^+$ of positive roots

\noindent \textbf{Output}: a (non-empty) set of positive roots or \emph{False}.
\begin{itemize}
\item[1.] Compute the set $\Phi_w = \{\alpha \in \Phi^+ \colon w\sigma(\alpha) < 0\}$.
\item[2.] Set $\Phi^+ \sm \Psi = \{\beta_1,\dots ,\beta_s\}$ with $s \geq 1$.
\item[3.] For $i$ running through $1, 2, \dots, s$ do:
\begin{itemize}
\item[3.1.] Set $\Psi_1^{(i)} = \Psi \cup \{\beta_i\}$ and $\Psi_2^{(i)} = \Psi_1^{(i)} \sm \Phi_w$.
\item[3.2.] Check whether the following conditions hold: (a) $\Psi_1^{(i)}$ and $\Psi_2^{(i)}$ are closed under addition; (b) for all $\alpha,\beta \in \Psi_1^{(i)}$ such that $\alpha+\beta \in\Phi^+$, one has $\alpha + \beta \in \Psi_2^{(i)}$; (c) $w\sigma(\Psi_2^{(i)}) \subseteq \Psi_1^{(i)}$.
\item[3.3] If (a)-(c) hold, return $\Psi_1^{(i)}$ and stop. Otherwise continue with the next $i$.
\end{itemize}
\item[4.] Return \emph{False}.
\end{itemize}

\noindent \emph{iterate\_steps} (procedure 2):

\noindent \textbf{Input}: an element $w \in W$, and $\Psi$, which is either a subset of $\Phi^+$ or \emph{False}.

\noindent \textbf{Output}: a (non-empty) set of positive roots or \emph{False} or \emph{True}.

\begin{itemize}
\item[1.] Compute the set $\Phi_w := \{\alpha \in \Phi^+ \colon w\sigma(\alpha) < 0\}$.
\item[2.] If $\Phi_w = \Phi^+$, return \emph{True} and stop.
\item[3.] If \emph{find\_candidate\_for\_one\_step}$(w, \Psi) =$ \emph{False}, return \emph{False} and stop.
\item[4.] If $\Psi = \Phi^+$, return \emph{True} and stop.
\item[5.] Set $\Psi' := $\emph{find\_candidate\_for\_one\_step}$(w,\Psi)$. Return \emph{iterate\_steps}$(w,\Psi')$.
\end{itemize}

\smallskip

To check if Lemma \cite[Lemma 5.7]{Ivanov_Cox_orbits} holds for an element $w \in W$, one runs the (recursive) procedure \emph{iterate\_steps} with arguments $w$ and $\Phi_w = \{\alpha \in \Phi^+ \colon w\sigma(\alpha) < 0\}$. The recursion stops after finitely many steps. If the final output is \emph{True}, the lemma holds. This holds true if $w$ is twisted Coxeter.

\begin{remark} Note that the final output \emph{True} of \emph{iterate\_steps}$(w,\Phi_w)$ is a sufficient but not a necessary condition for Lemma \cite[Lemma 5.7]{Ivanov_Cox_orbits} to hold for $w \in W$. In fact, there are (non-Coxeter) elements $w \in W$ for which \cite[Lemma 5.7]{Ivanov_Cox_orbits} holds true, but \emph{iterate\_steps}$(w,\Phi_w)$ outputs \emph{False}.
\end{remark}

\bibliography{bib_ADLV}{}
\bibliographystyle{amsalpha}

\end{document}